\documentclass[11pt,oneside,english,reqno]{amsart}
\usepackage[T1]{fontenc}
\usepackage[latin9]{inputenc}
\usepackage{color}
\usepackage{babel}
\usepackage{amstext}
\usepackage{float}
\usepackage{amsthm}
\usepackage{amssymb}
\usepackage{graphicx}
 
\usepackage{geometry}
\makeatletter
\numberwithin{equation}{section}
\numberwithin{figure}{section}
\theoremstyle{definition}
\newtheorem{thm}{\protect\theoremname}[section]
  \theoremstyle{definition}
  
  \theoremstyle{definition}
  \newtheorem{defn}[thm]{\protect\definitionname}
  \theoremstyle{definition}
  \newtheorem{lem}[thm]{\protect\lemmaname}
  \theoremstyle{remark}
  \newtheorem{rem}[thm]{\protect\remarkname}
  \theoremstyle{definition}
  \newtheorem{cor}[thm]{\protect\corollaryname}
  \theoremstyle{definition}
  \newtheorem*{thm*}{\protect\theoremname}
  \theoremstyle{definition}
  \newtheorem{prop}[thm]{\protect\propositionname}

\usepackage[latin9]{inputenc}
\usepackage{geometry}
\usepackage{amsmath}
\usepackage{amsthm}

\numberwithin{equation}{section}
\numberwithin{figure}{section}
\usepackage{enumitem}		
 \let\footnote=\endnote
\@ifundefined{lettrine}{\usepackage{lettrine}}{}

\theoremstyle{definition}
\newtheorem{thmx}{Theorem}

\def\d{\delta}

\def\L{\Lambda}

\def\D{\mathbb{D}}
\def\R{\mathbb{R}}
\def\C{\mathbb{C}}

\def\N{\mathbb{N}}
\def\Z{\mathbb{Z}}

\def\ep{\varepsilon}
\def\vphi{\varphi}

\usepackage{float}

\address{Department of Mathematics, The University of Chicago, Chicago, IL 60637, USA} 
\email{kihopark@math.uchicago.edu}

\def\A{\mathcal{A}}
\def\B{\mathcal{B}}

\def\L{\mathcal{L}}

\def\D{\mathcal{D}}
\def\F{\mathcal{F}}
\def\P{\mathsf{P}}
\def\R{\mathbb{R}}
\def\C{\mathbb{C}}
\def\CC{\mathcal{C}}

\def\I{\mathrm{I}}
\def\J{\mathrm{J}}
\def\K{\mathrm{K}}
\def\W{\mathcal{W}}

\def\M{\mathcal{M}}

\def\Md{M_{d \times d}(\R)}
\def\vps{\vphi^s}
\def\U{\mathcal{U}}
\def\V{\mathbb{V}}
\def\VV{\mathcal{V}}
\def\WW{\mathbb{W}}

\def\ep{\varepsilon}

\def\hol{H\"older }
\def\fb{fiber-bunched }
\def\mangle{\measuredangle}

\newcommand{\Wloc}{\mathcal{W}_{\text{loc}}}

\newcommand{\Sig}{\Sigma_T}
\newcommand{\slr}{\text{SL}_d(\R)}
\newcommand{\glr}{\text{GL}_d(\R)}

\makeatother
  \providecommand{\corollaryname}{Corollary}
  \providecommand{\definitionname}{Definition}
  \providecommand{\factname}{Fact}
  \providecommand{\lemmaname}{Lemma}
  \providecommand{\propositionname}{Proposition}
  \providecommand{\remarkname}{Remark}
  \providecommand{\theoremname}{Theorem}
\providecommand{\theoremname}{Theorem}

\begin{document}
\sloppy
\title{Quasi-multiplicativity of typical cocycles}
\author{Kiho Park}
\date{\vspace{-5ex}}
\begin{abstract}
We show that typical (in the sense of \cite{bonatti2004lyapunov} and \cite{avila2006simplicity}) \hol and fiber-bunched $\glr$-valued cocycles over a subshift of finite type are uniformly quasi-multiplicative with respect to all singular value potentials. We prove the continuity of the singular value pressure and its corresponding (necessarily unique) equilibrium state for such cocycles, and apply this result to repellers. Moreover, we show that the pointwise Lyapunov spectrum is closed and convex, and establish partial multifractal analysis on the level sets of pointwise Lyapunov exponents for such cocycles.
\end{abstract}
\date{}
\maketitle
\tableofcontents

\section{Introduction}
Given a finite set of $M_{d \times d}(\R)$ matrices $\mathsf{A}=\{A_1,\ldots,A_q\}$ and an infinite word $x^+ = x_0x_1x_2\ldots$ where each $x_j \in \{1,2,\ldots,q\}$, consider the products 
\begin{equation}\label{eq: product matrices}
A_{x_{n}}\ldots A_{x_0},
\end{equation}
for $n=1,2,\ldots$.
The study of such products naturally arises in many settings and has numerous applications. For instance, suppose each $A_i$ is contracting, and $T_i$ is an affine transformation of $\R^d$ whose linear part is $A_i$; that is, $T_i(x) = A_ix+r_i$ for some translation vector $r_i$. Then there exists a unique self-affine attractor $X\subset \R^d$ invariant under $\{T_1,\ldots, T_q\}$, in the sense that $X  =\bigcup\limits_{i=1}^q T_iX$; see \cite{hutchinson1981fractals}. The local geometry of the attractor $X$ depends on properties of the composition of the linear contractions \eqref{eq: product matrices}; for example, the Hausdorff dimension of $X$ is intimately related to the growth of the product \eqref{eq: product matrices} over all possible words $x^+$. See Remark \ref{rem: self-affine sets}.

Among many methods to analyze the product \eqref{eq: product matrices}, one is to study the limit (if it exists) of the following expression 
$$\lim\limits_{n \to \infty}\frac{1}{n}\log\|A_{x_{n-1}}\ldots A_{x_0}\|.$$
If the limit exists at $x^+ = x_0x_1\ldots$, we call it the \textit{pointwise Lyapunov exponent} of $x^+$, and it measures the asymptotic growth rate of the product \eqref{eq: product matrices}.
Once we put a standard metric on the space of all possible words (see Section 2; roughly, two words $x^+$ and $y^+$ are close if they agree along a long initial string), it is not hard to see that in general the pointwise Lyapunov exponent is a discontinuous function in $x^+$. Nonetheless, under mild assumptions on the matrices $\mathsf{A}$, the structure of the Lyapunov spectrum (i.e., the values of the pointwise Lyapunov exponent) is quite regular. For example, under the assumption that the matrices in $\mathsf{A}$ do not preserve a common proper subspace of $\R^d$ (i.e., $\mathsf{A}$ is irreducible), Feng \cite{feng2003lyapunov,feng2009lyapunov} showed that the spectrum is a closed interval.

The product of matrices \eqref{eq: product matrices} can be placed in a broader context.
To any dynamical system $f\colon X\to X$ and map $\A \colon X \to \Md$, we can associate a \textit{linear cocycle} $F_\A \colon X \times \R^d \to X \times \R^d$ given by
$$F_\A(x,v) = (fx,\A(x)v).$$
We say that $F_\A$ is \textit{generated} by $f$ and $\A$.
For $n \in \N$ and $x \in X$, we write $F_\A^n(x,v) = (f^nx,\A^n(x)v)$, where 
$$\A^n(x):= \A(f^{n-1}x)\ldots \A(fx)\A(x).$$
The definition of linear cocycle $F$ also extends to (not necessarily trivializable) vector bundles $\mathcal{E}$ over $X$ as a family of linear maps $F_x \colon \mathcal{E}_x \to \mathcal{E}_{fx}$ covering a base system $(X,f)$. 

When the base system is the left shift operator on a one-sided shift $\Sigma_q^+ = \{1,2,\ldots,q\}^{\N_0}$, then the map $\A \colon \Sigma_q^+ \to \Md$ defined by $x = (x_i)_{i \in \N_0} \mapsto A_{x_0}$ generates a linear cocycle $F_\A$. The cocycle $F_\A$ encodes the products \eqref{eq: product matrices} in the sense that $\A^n(x) = A_{x_{n-1}}\ldots A_{x_0}$, and it is an example of a locally constant cocycle (See Definition \ref{defn: loc constant} and Remark \ref{rem: loc constant}). 

Another natural class of linear cocycles comes from smooth dynamics. When the base system $f \colon M \to M$ is a smooth map or diffeomorphism of a closed Riemannian manifold $M$, the \textit{derivative cocycle} $Df$ is a cocycle generated by the map $\A(x) = D_xf \colon T_xM \to T_{fx}M$. More generally, for any $Df$-invariant sub-bundle $E \subset TM$, the derivative map restricted to $E$ gives rise to a linear cocycle $Df|_E$. If $f$ is uniformly hyperbolic (i.e., expanding or Anosov), then there exists a symbolic coding of $f$ by a subshift of finite type \cite{sinai1968markov}, \cite{bowen1970markov}. From such a coding, the derivative cocycle of a uniformly hyperbolic map can effectively be regarded as a linear cocycle over a subshift of finite type.


The main objects of interest in this paper are linear cocycles $F_\A$ over a subshift of finite type $(\Sigma,f)$ generated by $\glr$-valued functions on $\Sigma$. In particular, we study the thermodynamic formalism of such cocycles. Any $\A \colon \Sigma \to \glr$ defines a sequence of continuous functions $\{\vphi_{\A,n}\}_{n \in \N}$ on $\Sigma$ given by $$\vphi_{\A,n}(x) = \|\A^n(x)\|,$$ where $\|\cdot \|$ is the operator norm.
The submultiplicativity of the norm $\|\cdot \|$ implies that this sequence is \textit{submultiplicative} in the sense that for any $m,n \in \N$, 
$$0 \leq \vphi_{\A,n+m} \leq \left(\vphi_{\A,n}\circ f^m \right) \cdot \vphi_{\A,m}.$$
A submultiplicative sequence gives rise to a \textit{singular value potential} $\Phi_\A = \{\log \vphi_{\A,n}\}_{n\in \N}$. 
The singular value potential $\Phi_\A$ is an example of a \textit{subadditive potential} $\Psi = \{\log \psi_n\}_{n \in \N}$
which can be thought of as a generalization of the Birkhoff sum $S_n\psi$ for a potential $\psi \in C(\Sigma,\R)$. The usual thermodynamical notions of the pressure and the equilibrium states of a potential $\psi$ extend to subadditive potentials \cite{cao2008thermodynamic}.

In his fundamental work on Thermodynamic formalism, Bowen \cite{bowen1974some} showed that for any \hol potential $\psi$ on a mixing hyperbolic system such as $(\Sigma,f)$, there exists a unique equilibrium state for $\psi$, and that such equilibrium state has the Gibbs property. 

It is natural to ask if Bowen's theorem (with suitable generalizations) holds for subadditive potentials such as $\Phi_\A$. Unfortunately, the analogue of Bowen's theorem does not necessarily hold for general subadditive potentials \cite{feng2010equilibrium}. On the other hand, Bowen's theorem remains valid for singular value potentials of certain cocycles, including the cocycles generated by locally constant $\glr$-valued functions satisfying an extra assumption known as quasi-multiplicativity. Denoting the set of all admissible words of $\Sigma$ by $\L$, for any $\A\colon \Sigma \to \glr$ and $\I \in \L$, we define 
\begin{equation}\label{eq: A(I)}
\|\A(\I)\|:=\max\limits_{x \in [\I]}\vphi_{\A,|\I|}(x) = \max_{x \in [\I]}\|\A^{|\I|}(x)\|.
\end{equation}

We say $\A$ is \textit{quasi-multiplicative} if there exist $c>0$ and $k \in \N$ such that for any words $\I,\J \in \L$, there exists $\K = \K(\I,\J) \in \L$ with $|\K| \leq k$ such that $\I\K\J \in \L$ and 
\begin{equation}\label{eq: qm intro}
\|\A( \I\K\J) \| \geq c \|\A(\I)\| \|\A(\J)\|.
\end{equation}
Notice that quasi-multiplicativity of $\A$ resembles Bowen's specification property \cite{bowen1974some} in some respects.

For locally constant cocycles, there is a sufficient condition that guarantees quasi-multiplicativity.
We say that a locally constant $\glr$-valued function $\A$ is \textit{irreducible} if the image of $\A$ (which is necessarily a finite set of matrices) doesn't preserve a common proper subspace of $\R^d$. It is well-known that an irreducible locally constant cocycle is quasi-multiplicative \cite{feng2010equilibrium}, \cite{bochi2018equilibrium}. Hence, for such cocycles $F_\A$, there is a unique equilibrium state for the singular value potential $\Phi_\A$. Such equilibrium states often have applications in the dimension theory of fractals \cite{falconer1988subadditive}, \cite{bochi2018equilibrium}.

In this paper, we address the question of whether quasi-multiplicativity holds for more general cocycles beyond locally constant cocycles. It is not entirely clear what the natural counterpart to irreducibility might be for general cocycles. On the other hand, since quasi-multiplicativity is a typical feature of locally constant cocycles, it is reasonable to expect that quasi-multiplicativity holds for a more general class of cocycles with suitable assumptions.

We restrict our attention to \hol continuous and fiber-bunched (See Section 2 for precise definitions) cocycles, a class that contains the locally constant cocycles. The fiber-bunching assumption is an open condition which roughly says that the cocycle is nearly conformal. We denote the space of $\alpha$-\hol and \fb functions by $C^\alpha_b(\Sigma,\glr)$, viewed as a subset of $C^\alpha(\Sigma,\glr)$.

Our main result establishes that quasi-multiplicativity holds generically among these cocycles. More precisely, Bonatti and Viana in \cite{bonatti2004lyapunov} introduced the notion of {\em typical} cocycles among fiber-bunched cocycles
(see Definition \ref{defn: 1-typical}, \ref{defn: t-typical}, and \ref{defn: typical} for precise formulations). The set
$$\U := \{\A \in C^\alpha_b(\Sigma,\glr) \colon \A \text{ is typical}\}$$
is open in $C^\alpha_b(\Sigma,\glr)$, and Bonatti and Viana \cite{bonatti2004lyapunov} also proved that $\U$ is dense in $C^\alpha_b(\Sigma,\slr)$ and that its complement has infinite codimension.
\begin{thmx}\label{thm: A}
Every $\A \in \U$ is quasi-multiplicative. Moreover, the constants $c,k$ in \eqref{eq: qm intro} can be chosen uniformly in a neighborhood of $\A$ in $\U$. 
\end{thmx}

Theorem \ref{thm: A} follows from a more general result: for 	$\A \in \U$, Theorem \ref{thm: E} (see Section 2) gives simultaneous quasi-multiplicativity of the exterior product cocycles $\A^{\wedge t}$, $t \in \{1,\ldots, d-1\}$ with uniform constants $c$ and $k$. 

As an application, we prove the continuity of the subadditive pressure on $\U$. More precisely, there exists a natural generalization of the singular value potential $\Phi_\A^s$ for all $s \in [0,\infty)$ by considering the exterior product cocycles $\A^{\wedge t}$ (See Section 2). Denoting the pressure of $\Phi_\A^s$ by $\P(\Phi_\A^s)$, we establish the following continuity result using the uniform constants $c,k$ from Theorem \ref{thm: E}: 

\begin{thmx}\label{thm: B}
~
\begin{enumerate}
\item The map $(\A,s) \mapsto \P(\Phi_\A^s)$ is continuous on $\U \times [0,\infty)$.
\item For each $\A \in \U$ and $s \in [0,\infty)$, the singular value potential $\Phi_\A^s$ has a unique equilibrium state $\mu_{\A,s}$, which also varies continuously on $\U \times [0,\infty)$.
\end{enumerate}
\end{thmx}

Cao, Pesin, and Zhao \cite{cao2018dimension} recently proved a result that implies Theorem \ref{thm: B} (1). See Section 3 for further comments.

Theorem \ref{thm: B} has further applications in dimension theory of repellers. Given a repeller $\Lambda$ (see Definition \ref{defn: repeller}), one can associate a number $s(\Lambda)$ obtained as the unique zero of Bowen's equation. 
Such number $s(\Lambda)$ is an upper bound, and often a natural estimate, on the Hausdorff dimension of the repeller. In fact, there are many settings in which $s(\Lambda)$ is equal to the Hausdorff dimension. See Section 5 and a survey \cite{chen2010dimension} for more details on the number $s(\Lambda)$. From its definition, it follows that $s(\Lambda)$ varies upper semi-continuously under small perturbations of the repeller $\Lambda$. 
Using Theorem \ref{thm: B}, we prove a result on the continuity of $s(\Lambda)$:

\begin{thmx}\label{thm: C}
Let $M$ be a Riemannian manifold, and let $h \colon M \to M$ be a $C^{r}$ map with $r>1$. Suppose $\Lambda \subset M$ is a $\alpha$-bunched repeller defined by $h$ for some $\alpha\in (0,1)$ with $r-1>\alpha$. Then there exist a $C^1$-neighborhood $\VV_1$ of $h$ in $C^{r}(M,M)$ and  a $C^1$-open and $C^r$-dense subset $\mathcal{V}_2 \subset \mathcal{V}_1$ such that the map $$g \mapsto s(\Lambda_g)$$ is continuous on $\mathcal{V}_2$.
\end{thmx}

As another application of Theorem \ref{thm: E}, we extend and generalize Feng's result \cite{feng2003lyapunov,feng2009lyapunov} which states that the pointwise Lyapunov spectrum of an irreducible locally constant cocycle is a closed interval. By considering the exterior product cocycle $\A^{\wedge t}$, we define 
$$\lambda_t(x) :=\lim\limits_{n \to \infty} \frac{1}{n}\log\vphi^t(\A^n(x)),$$
if the limit exists, and set
$$\vec{\lambda}(x) := (\lambda_1(x),\ldots,\lambda_d(x)),$$
if $\lambda_t(x)$ exists for each $1 \leq t \leq d$. We define the \textit{pointwise Lyapunov spectrum} of $\A$ as 
$$L_{\A}:=\{\vec{\alpha} \in \R^d  \colon \vec{\alpha} = \vec{\lambda}(x) \text{ for some } x \in \Sigma\}.$$

\begin{thmx}\label{thm: D}
Let $\A \in \U$. Then $L_{\A}$ is a closed and convex subset of $\R^d$. 
\end{thmx}

Recall that a repeller $\Lambda$ is {\em conformal} if the derivative map $D_xh$ is a conformal transformation for every $x \in \Lambda$.
Combining Theorem \ref{thm: D} with the proof of Theorem \ref{thm: C}, we obtain the following corollary whose proof appears in Section 6.

\begin{cor}\label{cor: perturb conformal Lyap spectrum}
Let $\Lambda \subset M$ be a conformal repeller defined by a $C^r$ map $h \colon M \to M$ with $r>1$. Then there exists a $C^1$-neighborhood $\VV_1$ of $h$ in $C^r(M,M)$ and a $C^1$-open and $C^r$-dense subset $\VV_2$ of $\VV_1$ such that for every $g \in \VV_2$, the pointwise Lypapunov spectrum $L_{g}$ of $g|_{\Lambda_g}$ is a closed and convex subset of $\R^d$.   
\end{cor}

Finally, we also obtain partial multifractal results on the level sets of pointwise Lyapunov exponents (Corollary \ref{cor: multifractal}) by applying general results in \cite{feng2010lyapunov}.

The paper is organized as follows. In Section 2, we introduce the setting of our results and state the main theorem (Theorem \ref{thm: E}) of the paper.  
In Section 3, we survey relevant results in thermodynamic formalism for both additive and subadditive settings. In Section 4, we prove Theorem \ref{thm: E} in a more general setting. In Section 5, we prove Theorem \ref{thm: B} and \ref{thm: C}. In Section 6, we establish Theorem \ref{thm: D} and Corollary \ref{cor: perturb conformal Lyap spectrum}. Moreover, we discuss further applications of Theorem \ref{thm: E}, including the structure of the pointwise Lyapunov spectrum as well as some of its level sets.
\\\\
\noindent\textbf{Acknowledgements} The author is very grateful to his advisor Amie Wilkinson for her support and numerous helpful discussions. The author would also like to thank Clark Butler for sharing his insights and for pointing out an error in Section 3 of the original draft, and Aaron Brown for many helpful suggestions. 
Lastly, the author also thanks De-Jun Feng for his comments and Ping Ngai Chung for for improving the readability of the paper.

\section{Preliminaries and statement of results}
\subsection{Symbolic dynamics}
An \textit{adjacency matrix} $T$ is a square (0,1)-matrix.
A one-sided \textit{subshift of finite type} defined by a $q \times q$ adjacency matrix $T$ is a dynamical system $(\Sig^+,f)$ where $$\Sig^+ := \{(x_i)_{i \in \N_0} \colon x_i \in \{1,2,\ldots,q\} \text{ and }T_{x_i,x_{i+1}} = 1 \text{ for all }i \in \N_0\}$$ and $f$ is the left shift operator. Similarly, we define a two-sided subshift of finite type $(\Sig,f)$ where 
$$\Sig := \{(x_i)_{i \in \Z} \colon x_i \in \{1,2,\ldots,q\} \text{ and }T_{x_i,x_{i+1}} = 1 \text{ for all }i \in \Z\}.$$
Then $(\Sig,f)$ is the \textit{natural extension} of $(\Sig^+,f)$: denoting the projection from $\Sig$ onto $\Sig^+$ by $$\pi\colon \Sig \to \Sig^+,$$ each $x \in \Sig$ corresponds to one possible sequence of preimages of $\pi(x) \in \Sig^+$ under $f$.

We will always assume that the adjacency matrix $T$ is \textit{primitive}, meaning that there exists $N>0$ such that all entries of $T^N$ are positive. The primitivity of $T$ is equivalent to the mixing property of the corresponding subshift of finite type $(\Sig,f)$.

Fix $\theta \in (0,1)$ and endow $\Sig$ with the metric $d$ defined as follows: for $x = (x_i)_{i \in \Z},y = (y_i)_{i \in \Z} \in \Sig$, we have
$$d(x,y)  = \theta^k,$$ 
where $k$ is the largest integer such that $x_i = y_i$ for all $|i| < k$. Equipped with such metric, the subshift of finite type $(\Sig,f)$ becomes a hyperbolic homeomorphism.

An \textit{admissible word of length $n$} is a word $i_0\ldots i_{n-1}$ with $i_j \in \{1,\ldots,q\}$ such that $T_{i_j,i_{j+1}} = 1$ for all $0 \leq j \leq n-2$.
Let $\L$ be the collection of all admissible words. For $\I\in \L$, we denote its length by $|\I|$.
For each $n \in \N$, let $\L(n) \subset \L$ be the set of all admissible words of length $n$. For any $\I =  i_0\ldots i_{n-1} \in \L(n)$, we define the associated \textit{cylinder} by
$$[\I]=[ i_0\ldots i_{n-1}]:=\{y \in \Sig \colon y_j = i_j \text{ for all } 0 \leq j \leq n-1\}.$$ 
For $x \in \Sig$ and $n \in \N$, we similarly define $$[x]_n:=\{y \in \Sig \colon y_i = x_i \text{ for all } 0 \leq i \leq n-1\}.$$ 
Using the superscript $w$, for each $x\in \Sig$, we denote the word $x_0\ldots x_{n-1}$ by 
$$[x]_n^w := x_0\ldots x_{n-1} \in \L(n).$$

We define the \textit{local stable set} $\Wloc^s(x)$ \textit{of $x \in \Sig$} by 
$$\Wloc^s(x):=\{y \in \Sig \colon x_i = y_i \text{ for all } i \geq 0\}.$$
In other words, $y \in \Sig$ belongs to $\Wloc^s(x)$ if the forward orbit of $y$ exponentially shadows the forward orbit of $x$, meaning that $d(f^nx,f^ny) \leq \theta^{n+1}$ for all $n \geq 0$. We extend the definition to define the \textit{stable set} $\W^s(x)$ \textit{of $x \in \Sig$} by 
$$\W^s(x):=\{y \in \Sig \colon f^ny \in \Wloc^s(f^nx) \text{ for some } n\geq 0\}.$$ The (local) stable set of $f^{-1}$ is called the \textit{(local) unstable set $\W^u$ of $f$}.

For any $x,y\in \Sig$ with $x_0 = y_0$, we say $y$ is in the \textit{local neighborhood} of $x$. For such $x$ and $y$, the following bracket operation 
\begin{equation}\label{eq: bracket}
[x,y]:=\Wloc^s(x) \cap \Wloc^u(y) \in \Sig
\end{equation}
is well-defined. 
From the definition, $[x,y]$ is the unique point in the local neighborhood of $x$ and $y$ that exponentially shadows the orbit of $x$ in the future and the orbit of $y$ in the past.

Recall from the introduction that to any dynamical system $(X,f)$ and $\Md$-valued function $\A$ on $X$, we associate a linear cocycle $F_\A$.
It is clear from the definition of $\A^n(\cdot)$ that the following cocycle equation holds:
$$\A^{n+m}(x) := \A^n(f^mx)\A^m(x) \text{ for all } n,m\in \N.$$

If the base system $(X,f)$ is invertible and the image of $\A$ is a subset of $\glr$, then we extend the definition to define $\A^{0}(\cdot) \equiv I$ and $\A^{-n}(x) : = \big(\A^n(f^{-n}x)\big)^{-1}$ for $n \in \N$ 
such that the cocycle equation holds for all $n,m \in \Z$. 

\begin{defn}\label{defn: loc constant}
We say $\A \colon \Sig \to \Md$ is \textit{locally constant} if there exists $k\in \N$ such that $\A(x)$ depends only on the word $x_{-k}\ldots x_k \in \L(2k+1)$ for every $x = (x_i)_{i \in \Z} \in \Sig$. A \textit{locally constant cocycle} $F_\A$ is a cocycle whose generator $\A$ is locally constant. \end{defn}

\begin{rem}\label{rem: loc constant}
For any locally constant function $\A$ on $\Sig$, there exists a re-coding of $\Sig$ to another subshift of finite type $\widetilde{\Sigma}_T$ such that $\A$ is carried to a function on $\widetilde{\Sigma}_T$ depending only on the 0-th entry $x_0$ of $x = (x_i)_{i\in \Z} \in \widetilde{\Sigma}_T$.

For simplicity, we assume that all locally constant functions considered in this paper are functions that depend only on the 0-th entry.
\end{rem}

\subsection{Holonomies and fiber-bunched cocycles}
Let $\A$ be an $\alpha$-\hol $\glr$-valued function on $\Sig$, meaning that there exists $C>0$ such that for all $x,y \in \Sig$,
$$\|\A(x)-\A(y)\|  \leq Cd(x,y)^\alpha,$$
where $\|\cdot \|$ is the standard operator norm. 

\begin{defn}\label{defn: hol rel}
A \textit{local stable holonomy} for $F_\A$ is a family of matrices $H^{s}_{x,y} \in \glr$ defined for any $x,y\in \Sig$ with $y \in \Wloc^s(x)$ such that
\begin{enumerate}
\item $H^s_{x,x} = I$ and $H^s_{y,z}\circ H^s_{x,y} = H^s_{x,z}$ for any $y,z \in \Wloc^s(x)$,
\item $\A(x) = H^s_{fy,fx} \circ \A(y) \circ H^s_{x,y}$,
\item $H^s\colon (x,y) \mapsto H^s_{x,y}$ is continuous.
\end{enumerate}
A \textit{local unstable holonomy} $H^u_{x,y}$ is likewise defined for $y \in \Wloc^u(x)$ satisfying the analogous properties above.
\end{defn}

We say that a stable/unstable holonomy $H^{s/u}$ is \textit{uniformly continuous} if for any $\ep>0$, there exists $\d>0$ such that for any $y \in \W^{s/u}(x)$, we have
$$d(x,y) \leq \d \implies \|H^{s/u}_{x,y} - I\| \leq \ep.$$

\begin{defn}
An $\alpha$-\hol function $\A$ is \textit{fiber-bunched} if for all $x\in \Sig$, we have
$$\|\A(x)\|\|\A(x)^{-1}\| \theta^\alpha <1,$$
where $\theta$ is the hyperbolicity constant defining the metric on the base $\Sig$. 

We denote the space of $\alpha$-\hol and \fb functions by $C^\alpha_b(\Sig,\glr)$, and say the cocycle $F_\A$ is \fb if its generator $\A$ belongs to $C^\alpha_b(\Sig,\glr)$.
\end{defn}

By projectivizing the action on the fibers, $F_\A$ gives rise to the \textit{projective cocycle} $\mathbb{P}(F_\A) \colon \Sig \times \mathbb{P}(\R^d) \to \Sig \times \mathbb{P}(\R^d)$. Then the fiber-bunching condition is equivalent to the condition that the rate of expansion (respectively, contraction) of the projective cocycle $\mathbb{P}(F_\A)$ at every point $x \in \Sig$ is bounded above by $1/\theta^{\alpha}$ (respectively, below by $\theta^\alpha$). In particular, the \hol continuity and the fiber-bunching assumption on $\A \in C^\alpha_b(\Sig,\glr)$ together ensure the convergence of the \textit{canonical stable/unstable holonomy} $H^{s/u}_{x,y}$: for any $y \in \Wloc^{s/u}(x)$, 
\begin{equation}\label{eq: hol}
H^s_{x,y} :=\lim\limits_{n \to \infty} \A^n(y)^{-1}\A^n(x) ~\text{ and }~H^u_{x,y}:= \lim\limits_{n \to -\infty} \A^n(y)^{-1}\A^n(x).
\end{equation}
See \cite{kalinin2013cocycles} or \cite{bonatti2003genericite} for details. 

A cocycle may admit multiple holonomies. However, when the cocycle is fiber-bunched, the canonical holonomies are unique in the sense that they are the only holonomies varying \hol continuously in the base points \cite{kalinin2013cocycles} with the same \hol exponent $\alpha$: there exists $C>0$ such that  
\begin{equation}\label{eq: hol holder}
\|H^{s/u}_{x,y} - I\| < Cd(x,y)^{\alpha},
\end{equation}
for any $y \in \Wloc^{s/u}(x)$. In particular, the canonical holonomies are uniformly continuous. We will always work with the canonical holonomies for fiber-bunched cocycles. 

\begin{rem}
It is worth noting a special family of cocycles trivially admitting the canonical holonomies. For any locally constant $\glr$-valued function $\A$, the canonical holonomies of $F_\A$ from \eqref{eq: hol} converge to the identity and satisfy the properties listed in Definition \ref{defn: hol rel}. 

The canonical holonomies of a fiber-bunched cocycle identify the fibers over points on the same (local) stable or unstable set, similar to how fibers over two nearby points can be trivially identified for locally constant cocycles.
\end{rem}

Using (2) of Definition \ref{defn: hol rel}, we can extend the definition of the local stable holonomy to the global stable holonomy $H^s_{x,y}$ for $y \in \W^s(x)$ not necessarily in the local stable set of $x$:
$$H^s_{x,y}:= \A^n(y)^{-1} H^s_{f^nx,f^ny}\A^n(x),$$
for some large enough $n \in \N$ so that $f^ny \in \Wloc^s(f^nx)$. We can likewise define the global unstable holonomy.  

A point $z \in \Sig$ is a \textit{homoclinic point} of a periodic point $p$ if $z \in \W^s(p) \cap \W^u(p) \setminus \{p\}$. The homoclinic points of $p$ are characterized as the points other than $p$ whose orbits synchronysly approach the orbit of $p$ in both forward and backward time. For a hyperbolic system such as $(\Sig,f)$, the set of homoclinic points of any periodic point is dense in $\Sig$.

\subsection{Typical cocycles}
We now formulate the assumptions building up to the main theorem. Consider any periodic point $p$ and a homoclinic point $z \in \W^s(p)\cap \W^u(p) \setminus{p}$. We define the \textit{holonomy loop} $\psi_p^z$ as the composition of the unstable holonomy from $p$ to $z$ and the stable holonomy from $z$ to $p$:
\begin{equation}\label{eq: original psi_p^z}
\psi_p^z  := H^s_{z,p}\circ H^u_{p,z}.
\end{equation}

The following definition is a slight weakening of the assumptions of Theorem 1 in \cite{bonatti2004lyapunov}. See Remark \ref{rem: BV comments}.
\begin{defn}\label{defn: 1-typical}
Let $\A \in C^\alpha_b(\Sig,\glr)$ and $H^{s/u}$ be the canonical holonomies for $F_\A$. We say that $\A$ is \textit{1-typical} if it satisfies the following two extra conditions:
\begin{enumerate}
\item[(A0)] there exists a periodic point $p$ such that $P:=\A^{\text{per}(p)}(p)$ has simple real eigenvalues of distinct norms. Let $\{v_i\}_{1\leq i \leq d}$ be the eigenvectors of $P$.
\item[(B0)] there exists a homoclinic point $z$ of $p$ such that $\psi_p^z$ twists the eigenvectors of $P$ into general position: 
for any $1 \leq i,j \leq d$, the image $\psi_{p}^z(v_i)$ does not lie in any hyperplane $\WW_j$ spanned by all eigenvectors of $P$ other than $v_j$. Equivalently, the coefficients $c_{i,j}$ in 
$$\psi_{p}^z(v_i) = \sum\limits_{1 \leq j \leq d}c_{i,j} v_j,$$
are nonzero for all $1 \leq i,j \leq d$.
\end{enumerate}
\end{defn}
\begin{rem}
We will often refer (A0) and (B0) by \textit{pinching} and \textit{twisting} conditions, respectively. 
\end{rem}


For each $1 \leq t \leq d$, we denote by $\A^{\wedge t}$ the action of $\A$ on the exterior product $(\R^d)^{\wedge t}$.
See subsection \ref{subsection: exterior algebra} for basic properties of the exterior product.
From the canonical holonomies $H^{s/u}$ for $F_\A$, the cocycles generated by $\A^{\wedge t},~1\leq t\leq d$ also admit stable and unstable holonomies, namely $(H^{s/u})^{\wedge t}$. So, for a 1-typical function $\A$, we consider similar conditions appearing in Definition \ref{defn: 1-typical} on $\A^{\wedge t}$.

\begin{defn}\label{defn: t-typical} Let $\A $ be 1-typical.
For $2 \leq t \leq d-1$, we say $\A$ is \textit{t-typical} if the same points $p,z \in \Sig$ from Definition \ref{defn: 1-typical} satisfy 
\begin{enumerate}
\item[(A1)] all the products of $t$ distinct eigenvalues of $P$ are distinct;
\item[(B1)] the induced map $(\psi_p^z)^{\wedge t} $ on $(\R^d)^{\wedge t}$ satisfies the analogous statement to (B0) from Definition \ref{defn: 1-typical} with respect to the eigenvectors $\{v_{i_1}\wedge \ldots \wedge v_{i_t}\}_{1\leq i_1<\ldots<i_t \leq d}$ of $P^{\wedge t}$.
\end{enumerate}
\end{defn}

\begin{rem}
Notice that the definition of $t$-typicality only asks for (A1) and (B1); the definition does not require that $\A^{\wedge t}$ also be fiber-bunched. 

On the other hand, we will use the fact that the stable and unstable holonomies $(H^{s/u})^{\wedge t}$ are uniformly continuous. This follows from the \hol continuity \eqref{eq: hol holder} of the canonical holonomies $H^{s/u}$ for $F_\A$.   
\end{rem}

\begin{defn}\label{defn: typical}
We say $\A$ is \textit{typical} if $\A$ is $t$-typical for all $1\leq t \leq d-1$. We denote $\U \subset C^\alpha_b(\Sig,\glr)$ to be the set of all typical functions.
\end{defn}
\begin{rem}\label{rem: BV comments}
A few comments regarding the assumptions are in order. First, similar (but slightly stronger) assumptions are introduced in Bonatti and Viana \cite{bonatti2004lyapunov} as a sufficient condition to establish the simplicity of the Lyapunov exponents of $F_\A$ for any ergodic $f$-invariant measure with continuous local product structure. Their setting is $\slr$-valued cocycles, and they also show that $\U$ is open and dense in $C^\alpha_b(\Sig,\slr)$. We remark that the difference in the settings ($\slr$ for \cite{bonatti2004lyapunov} and $\glr$ for this paper) does not cause any issues in translating the relevant statements and results from \cite{bonatti2004lyapunov} to this paper.

Avila and Viana in \cite{avila2006simplicity} improved the result by removing the assumption on the exterior powers and allowing the number of symbols of $\Sig$ to be countably infinite. Under many different settings, such assumptions serve as checkable conditions to establish the simplicity of the Lyapunov exponents; see \cite{backes2016simplicity} for instance. Our twisting condition (B1) on $\psi_p^z$ is weaker than both \cite{bonatti2004lyapunov} and \cite{avila2006simplicity}, but we still require the assumption on the simplicity of the eigenvalues of $P^{\wedge t}$ for all $1 \leq t \leq d-1$. In all cases, such assumptions are satisfied by an open and dense subset $\U$ of maps in $C^\alpha_b(\Sig,\glr)$, and the complement of $\U$ has infinite codimension.
\end{rem}

\begin{rem}\label{rem: homoclinic pts}
If $z$ is a homoclinic point of $p$, then $f^rz$ for any $r \in \Z$ is also a homoclinic point of $p$. Their holonomy loops are conjugated by $P^r$:
$$P^r\psi_p^z = \psi_p^{f^rz} P^r.$$ 
It then follows that if the twisting condition (B0) holds at $z$, then it also holds at $f^rz$.

Suppose $z$ is a homoclinic point of $p$ on $\Wloc^u(p)$ and $f^\ell z \in \Wloc^s(p)$ for some $\ell \in \N$. From Proposition \ref{defn: hol rel}, $\psi_p^z$ satisfies the relation
\begin{equation}\label{eq: psi_p^z}
\psi_p^z=P^{-\ell} \circ H^s_{f^\ell z,p}\circ\A^\ell(z) \circ H^u_{p,z}.
\end{equation}
\end{rem}

\subsection{Quasi-multiplicativity and the main theorem}

In order to state the main theorem, we need to introduce the notion of quasi-multiplicativity. Recalling that $\L$ is the set of all admissible words,
a function $\psi \colon \L \to \R_{\geq 0}$ is \textit{submultiplicative} if 
$$\psi(\I) \psi(\J) \geq \psi(\I\J)$$
for all $\I,\J \in \L$ with $\I\J \in \L$. 
Let $\D$ be the set of non-negative and submultiplicative functions on $\L$: 
$$\D =\{\psi \colon \L \to \R_{\geq 0}  \colon \psi \text{ is submultiplicative}\}.$$

\begin{defn}\label{defn: qm}
A non-negative and submultiplicative function $\psi \in \D$ is \textit{quasi-multiplicative} if there exist constants $c>0$ and $k\in\N$ such that for any words $\I,\J \in \L$, 
there exists $\K = \K(\I,\J) \in \L$ with $|\K| \leq k$ such that $\I\K\J \in \L$ and 
\begin{equation}\label{eq: qm}
\psi(\I\K\J) \geq c \psi(\I) \psi(\J).
\end{equation}
\end{defn}


\begin{rem}
Suppose $\psi \colon \L \to \R_{\geq 0}$ is not submultiplicative, but still satisfies the following weaker property: there exists $C\geq 1$ such that for all $\I,\J \in \L$, we have
\begin{equation}\label{eq: quasi-submult}
 C \psi(\I)\psi(\J) \geq \psi(\I\J).
\end{equation}

Then, $ C\psi$ is submultiplicative, and we can consider quasi-multiplicativity of the function $C\psi$. For such $\psi$, \eqref{eq: qm} and \eqref{eq: quasi-submult} together resemble the usual notion of a quasimorphism for the function $\log \psi$. 

However, we are mainly interested in the singular value potentials (see Section 3 for the definition), which are automatically submultiplicative. Hence, we have stated the definition of quasi-multiplicativity for submultiplicative functions. 
\end{rem}

Consider a family of quasi-multiplicative functions on $\L$. If they all admit uniform constants $c>0$ and $k \in \N$ as well as the common connecting word $\K = \K(\I,\J)$ for any $\I,\J \in \L$, then it would make sense to consider such functions as being simultaneously quasi-multiplicative.

\begin{defn}
Let $\mathcal{I}$ be an index set. A family of non-negative and submultiplicative functions $\psi^{(i)}\in\D$, $i \in \mathcal{I}$ are \textit{simultaneously quasi-multiplicative} if there exist constants $c>0$ and $k\in\N$ such that for any words $\I,\J \in \L$, there exists a word $\K = \K(\I,\J) \in \L$ with $|\K| \leq k$ such that $\I\K\J \in \L$ and 
$$\psi^{(i)}(\I\K\J) \geq c \psi^{(i)}(\I) \psi^{(i)}(\J),$$
for all $i \in \mathcal{I}$.
\end{defn}
 
We are most interested in quasi-mutiplicativity of the singular value functions related to a cocycle $F_\A$.
The \textit{singular values} of $A \in M_{d \times d}(\R)$ are eigenvalues of $\sqrt{A^*A}$. We define the \textit{singular value function} $\vps \colon \Md \to \R$ with parameter $s\geq 0 $ as follows:
$$\vps(A) = \begin{cases} 
      \alpha_1(A)\ldots\alpha_{\lfloor s \rfloor}(A)\alpha_{\lceil s \rceil}(A)^{\{s\}} & 0\leq s \leq d ,\\
      |\det(A)|^{s/d} & s>d,
   \end{cases}$$
where $\alpha_1(A) \geq \ldots \geq \alpha_d(A) \geq 0$ are the singular values of $A$. 
It is well-known that $\vps$ is submultiplicative for all $s$: for any $A,B \in \Md$ and $s \in [0,\infty)$,
\begin{equation}\label{eq: submult}
 \vps(A) \vps(B) \geq \vps(AB).
\end{equation}
Moreover, the function $(A,s) \mapsto \vps(A)$ is upper semi-continuous, and has a discontinuity at $s=k \in \N$ only when there is a jump in the singular values of the form $\alpha_{k-1}(A)>\alpha_k(A) = 0$. In particular, if $A$ takes value in $\glr$, then $\vps(A)$ is continuous in both $A$ and $s$. 

For any $\A \colon \Sig \to \glr$ and $s \geq 0$, we can associate them to a non-negative function (which we also call the singular value function) on $\L$ denoted by $\widetilde{\vphi}^s_\A \in \D$ as follows: for any $n \in \N$ and $\I \in \L(n)$, we define
\begin{equation}\label{eq: simplification 1}
\widetilde{\vphi}^s_\A(\I) := \max\limits_{x \in [\I]} \vps( \A^n(x)).
\end{equation}
Notice that this definition is similar to how we define $\|\A(\I)\|$ in the introduction \eqref{eq: A(I)}.
From the submultiplicativity of $\vphi^s$, it follows that $\widetilde{\vphi}^s_\A$ is also submultiplicative. We are now ready to state the main theorem of the paper.
\begin{thmx}\label{thm: E}
Let $\A \in \U$ be typical. Then the singular value functions $\widetilde{\vphi}^s_\A$ with $s  \in [0,d]$ are simultaneously quasi-multiplicative. Moreover, the constants $c,k$ can be chosen uniformly in a neighborhood of $\A$ in $\U$.
\end{thmx}

\begin{rem}\label{rem: comments on qm main}
We make a few remarks on Theorem \ref{thm: E}.
In the statement of Theorem \ref{thm: E}, we note that the parameter $s$ of the singular value function $\widetilde{\varphi}^s_\A$ varies only within the range $[0,d]$. This is mainly due to two reasons: first, the singular value function $\varphi^s$ takes a particularly simple form when $s > d$, and second, it suffices to consider $s \in [0,d]$ in the applications appearing in Section 5. 
If the parameter $s$ were to vary within $[0,s_0]$ for some $s_0 \in \R^+_0$, then the theorem still remains true except that the constant $c$ will have to change depending on $s_0$.
We also note that the theorem is not necessarily true if we consider simultaneous quasi-multiplicativity of $\widetilde{\varphi}^s_\A$ over the range $[0,\infty)$ for the parameter $s$. See Remark \ref{rem: s_0 for all R+} for further comments.

Lastly, note that Theorem \ref{thm: A} is an immediate corollary of Theorem \ref{thm: E}. The proof of Theorem \ref{thm: E} appears in Section 4. 
\end{rem}

\section{Thermodynamic formalism}

\subsection{Additive thermodynamic formalism}
Let $f$ be a continuous map on a compact metric space $X$. A \textit{potential} on $X$ is a continuous function $\psi \colon X \to \R$. 

A subset $E \subset X$ is $(n,\ep)$\textit{-separated} if every pair of distinct $x,y\in E$ satisfies $$d_n(x,y):=\max\limits_{0 \leq i \leq n-1} d(f^ix,f^iy) \geq \ep.$$ Using $(n,\ep)$-separated subsets, we can define a thermodynamical object called the \textit{pressure} $\P(\psi)$ of $\psi$ as follows: 
$$\P(\psi) := \lim\limits_{\ep \to 0} \limsup\limits_{n \to \infty} \frac{1}{n}\log \sup\Big\{ \sum\limits_{x \in E}  e^{S_n\psi(x)} \colon E \text{ is an }(n,\ep)\text{-separated subset of }X \Big\},$$
where $S_n\psi = \psi+\psi\circ f +\ldots +\psi\circ f^{n-1}$.

When $\psi \equiv 0$, the pressure $\P(0)$ is equal to the \textit{topological entropy} $h(f)$, which measures the complexity of the system $(X,f)$. 

Denoting the space of $f$-invariant probability measures on $X$ by $\M(f)$, the pressure satisfies the \textit{variational principle}:
$$\P(\psi) = \sup\Big\{h_\mu(f)+ \int \psi d\mu \colon \mu \in \M(f) \Big\},$$
where $h_\mu(f)$ is the measure-theoretic entropy. See \cite{walters2000introduction}.

Any invariant measure $\mu \in \M(f)$ achieving the supremum in the variational principle is called an \textit{equilbrium state} of $\psi$. 
If the entropy map $\mu \mapsto h_{\mu}(f)$ is upper semi-continuous, then any potential has an equilibrium state. 
However, the existence, the finiteness, or the uniqueness of the equilibrium state for a given potential is a subtle question that depends on the system $(X,f)$ as well as the potential $\psi$.

On the other hand, there are specific settings where such questions have an affirmative answer. When $(X,f)$ is a mixing hyperbolic system such as $(\Sig,f)$, and the potential $\psi$ is H\"older, then the result of Bowen \cite{bowen1974some} states that there exists a unique equilibrium state $\mu_\psi$, which has the Gibbs property.
\begin{prop}\label{prop: bowen} 
Let $(\Sig,f)$ be a mixing subshift of finite type, and $\psi$ be H\"older continuous. Then there exists a unique equilibrium state $\mu_\psi$ of $\psi$, characterized as the unique $f$-invariant measure satisfying the Gibbs property: there exists $C \geq 1$ such that for any $n\in \N$ and $\I \in \L(n)$, we have
\begin{equation}\label{eq: gibbs ineq}
C^{-1} \leq  \frac{\mu_{\psi}(\I)}{e^{-n\P(\psi)+S_n
\psi(x)}}\leq C
\end{equation}
for any $x \in \I$.
\end{prop}

\begin{rem}\label{rem: gibbs bdd distortion}
One necessary condition for the Gibbs property \eqref{eq: gibbs ineq} to hold is that the variation within each cylinder should be uniformly bounded: there exists a constant $C\geq 0$ such that for every $n \in \N$ and $\I \in \L(n)$,
\begin{equation}\label{eq: bounded distortion for additive potential}
|S_n\psi(x) -S_n\psi(y)|\leq C
\end{equation}
for every $x,y\in \I$. We denote this property by \textit{bounded distortion}.
\end{rem}

In the setting of Bowen's theorem, the hyperbolicity of the system and the \hol regularity of the potential guarantee the bounded distortion property. 

\subsection{Subadditive thermodynamic formalism}
The additive theory of thermodynamic formalism extends to the subadditive theory with suitable generalizations. 
A sequence of continuous functions $\{\psi_n\}_{n \in \N}$ on $\Sig$ is \textit{submultiplicative} if each $\psi_n$ is a non-negative function on $\Sig$ satisfying
$$0 \leq \psi_{m+n} \leq \psi_n \cdot  \psi_m\circ f^n, \text{ for all }m,n \in \N.$$
If we set $\Psi = \{\log \psi_n\}_{n \in \N}$, then $\Psi$ becomes a \textit{subadditive} sequence of functions on $\Sig$. We will consider such $\Psi$ obtained from a submultiplicative sequence $\{\psi_n\}_{n \in \N}$ as a \textit{subadditive potential} on $\Sig$. A natural example of a subadditive potential is a \textit{singular value potential} of a continuous $\glr$-valued function $\A$ on $\Sig$: for $s \geq 0$, we define
$$\Phi_\A^s:=\{\log\vps(\A^n(\cdot))\}_{n \in \N}.$$

We define the \textit{subadditive pressure} of a subadditive potential $\Psi = \{\log \psi_n\}_{n \in \N}$ as
\begin{equation}\label{eq: pressure}
\P(\Psi) := \lim\limits_{\ep \to 0}\limsup\limits_{n \to  \infty} \frac{1}{n} \log \sup\Big\{ \sum\limits_{x \in E}  \psi_n(x) \colon E \text{ is an }(n,\ep)\text{-separated subset of }\Sig \Big\},
\end{equation}
where the existence of the limit is guaranteed from the subadditivity of $\Psi$. 

There are a few different generalizations of the additive notion of the pressure to the subadditive setting: Barreira \cite{barreira1996non} defines the subadditive pressure by open covers while Cao, Feng, and Huang \cite{cao2008thermodynamic} define it using Bowen balls. Our definition of the subadditive pressure \eqref{eq: pressure} is based on \cite{cao2008thermodynamic}. See also \cite{falconer1988subadditive}. It is not known whether two definitions of the subadditive pressure are equal in general, but there are known settings in which two definitions agree. In particular, it is shown in \cite{cao2008thermodynamic} that two notions are equivalent when the entropy map $\mu \mapsto h_\mu(f)$ is upper semi-continuous, which includes our setting of mixing subshifts of finite type $(\Sig,f)$.

Cao, Feng, and Huang \cite{cao2008thermodynamic} also establish the \textit{subadditive variational principle}: 
\begin{equation}\label{eq: var prin}
\P(\Psi) = \sup \Big\{ h_\mu(f)+\F(\Psi,\mu) \colon \mu \in \M(f), ~\F(\Psi,\mu) \neq -\infty \Big\},
\end{equation}
where $$\F(\Psi,\mu) := \lim\limits_{n \to \infty} \frac{1}{n} \int \log\psi_n ~d\mu = \inf\limits_{n  \to \infty}\frac{1}{n} \int \log\psi_n ~d\mu,$$
whose limit is again guaranteed from the subadditivity of $\Psi$. 

Similar to the additive setting, any invariant measure $\mu \in \M(f)$ achieving the supremum in \eqref{eq: var prin} is called an \textit{equilibrium state} of $\Psi$. Also, at least one equilibrium state necessarily exists for any subadditive potential $\Psi$ if the entropy map $\mu \mapsto h_\mu(f)$ is upper semi-continuous \cite{feng2011equilibrium}.
See also \cite{kaenmaki2003natural}. 

Recall that $\D$ is the set of non-negative and submultiplicative functions on $\L$.
For any submultiplicative sequence $\{\psi_n\}_{n \in \N}$ on $\Sig$, we associate a function $\psi \in \D$ similar to \eqref{eq: A(I)} and \eqref{eq: simplification 1}: for any $n\in \N$ and $\I \in \L(n)$, let
\begin{equation}\label{eq: simplification 2}
\psi(\I):=\max\limits_{x \in [\I]}\psi_n(x).
\end{equation}
Hence, we can extend the notion of quasi-multiplicativity to submultiplicative sequences as follows.
\begin{defn}\label{defn: qm for subadditive}
We say that a submultiplicative sequence of continuous functions $\{\psi_n\}_{n \in \N}$ on $\Sig$ (or its associated subadditive potential $\Psi = \{\log \psi_n\}_{n \in \N}$) is \textit{quasi-multiplicative} if the function $\psi \in \D$ obtained from $\{\psi_n\}_{n \in \N}$ by $\eqref{eq: simplification 2}$ is quasi-multiplicative in the sense of Definition \ref{defn: qm}.

We say that $\A \colon \Sig \to \glr$ is \textit{quasi-multiplicative} if its singular value potential $\Phi_\A^1$ is quasi-multiplicative. This agrees with the definition of quasi-multiplicativity \eqref{eq: qm intro} of $\A$ from the introduction. 
\end{defn}

Conversely, for any $\psi \in \D$, we can associate a subadditive potential $\Psi = \{\log \psi_n\}_{n \in \N}$ in an obvious way:  
\begin{equation}\label{eq: psi_n from psi}
\psi_n(x) := \psi([x]_n^w),
\end{equation}
Hence, we can consider the pressure and the equilibrium states of functions in $\D$.

In the following subsection, we will discuss a sufficient condition for quasi-multiplicativity of locally constant cocycles as well as some of its consequences.

%

\subsection{Bowen's theorem for subadditive potentials}
In this subsection, we show that Bowen's theorem (Proposition \ref{prop: bowen}) remains to hold (with suitable generalizations) for subadditive potentials with quasi-multiplicativity.

For subadditive potentials, equilibrium states are often not unique, and such examples can be found where the subadditive potential is given by the singular value potential of some $\Md$-valued function. See \cite{feng2010equilibrium}, for instance.


More specifically, consider a subadditive potential $\Psi$ obtained from $\psi \in \D$ by \eqref{eq: psi_n from psi}. Alternatively, we can characterize such $\Psi = \{\log \psi_n\}_{n\in\N}$ by the condition that for any $n \in \N$ and $\I \in \L(n)$, 
\begin{equation}\label{eq: psi for loc const}
\psi_n(x) = \psi_n(y) \text{ for all } x,y\in [\I].
\end{equation}
Such $\Psi$ can be thought of as a subadditive potential with zero variation within cylinders. An example of such $\Psi$ is the singular value potential $\Phi_\A^s$ for locally constant $\glr$-valued functions $\A$. 

The main consequence of quasi-multiplicativity of $\psi\in \D$ is the uniqueness of the equilibrium state for the corresponding subadditive potential $\Psi$. 
\begin{prop}\cite[Theorem 5.5]{feng2011equilibrium}\label{prop: unique eq as gibbs} 
Let $\psi \in \D$ be quasi-multiplicative. Then the associated subadditive potential $\Psi =\{ \log\psi_{n}\}_{n \in \N}$ obtained from $\psi$ as in \eqref{eq: psi_n from psi} has a unique equilibrium state $\mu_\psi \in \M(f)$. Such $\mu$ is ergodic and has the following Gibbs property: there exists $C\geq 1$ such that 
\begin{equation}\label{eq: gibbs for subadditive}
C^{-1} \leq  \frac{\mu_\psi(\I)}{e^{-n\P(\Psi)}\psi(\I)} \leq C
\end{equation}
for any $n \in \N$ and $\I \in \L(n)$.
\end{prop}
\begin{rem}
In Feng \cite[Theorem 5.5]{feng2011equilibrium}, this result is proved for one-sided subshifts of finite type. This generalizes easily to two-sided subshifts of finite type. We briefly summarize the proof, which is similar to Bowen's original proof. 
Define a sequence of probability measures $\nu_n$ on the $\sigma$-algebra generated by $n$-cylinders by 
$\nu_n (\I) = \psi (\I) /\sum\limits_{\J \in \L(n)} \psi(\J)$, and considers any subsequential weak-$*$ limit $\mu \in \M(f)$ of the new sequence of probability measures $\mu_n = \frac{1}{n}\sum\limits_{i=0}^{n-1} f_*^i\nu_n$.
Quasi-multiplicativity then gives the Gibbs property as well as the ergodicity on $\mu$. In fact, the sequence $\mu_n$ actually converges to $\mu$ (i.e., a subsequential limit is an actual limit), and $\mu$ is the unique equilibrium state of $\Psi$. 
The same proof readily extends to our setting of two-sided subshifts of finite type.
\end{rem}

The following remark provides a criterion to establish quasi-multiplicativity for a locally constant $\glr$-valued function.

\begin{rem}\label{rem: irreducibility}
Recall that an $\glr$-valued function $\A$ on $\Sig$ is irreducible if there does not exist a proper subspace of $\R^d$ preserved under the image of $\A$ (which is necessarily a finite set of matrices). It is well-known that irreducibility of a locally constant function implies quasi-multiplicativity. See \cite{feng2009lyapunov}.

The typicality assumption in Theorem \ref{thm: E} is related to irreducibility of locally constant cocycles because a locally constant and typical cocycle is necessarily irreducible. This follows because any $\A$-invariant subspace has to be a span of some eigendirections of $\A(p)$; if $\A$ is not irreducible, then then $\A$ would not satisfy the twisting condition (B0) and consequently fail to be typical.
\end{rem}

\subsection{Subadditive potentials with bounded distortion}
In the previous subsection, we saw that quasi-multiplicativity of $\psi\in \D$ is a sufficient condition for Bowen's theorem (Proposition \ref{prop: unique eq as gibbs}) to hold for a subadditive potential $\Psi$ with zero variation within cylinders (i.e., satisfying \eqref{eq: psi for loc const}). 

In this subsection, we show that Bowen's theorem in the subadditive setting (Proposition \ref{prop: unique eq as gibbs}) can be considered on a bigger class of subadditive potentials than those satisfying \eqref{eq: psi for loc const}. Such class consists of subadditive potentials $\Psi = \{\log\psi_n\}_{n \in \N}$ with \textit{bounded distortion}: there exists $C\geq 1$ such that for any $n \in \N$ and $\I \in \L(n)$, we have
\begin{equation}\label{eq: bdd distortion general}
C^{-1} \leq \frac{\psi_n(x)}{\psi_n(y)} \leq C
\end{equation}
for any $x,y \in [\I]$.

As noted in Remark \ref{rem: gibbs bdd distortion}, in order to generalize the Gibbs property \eqref{eq: gibbs ineq} to the general subadditive setting, one necessary condition on the subadditive potential $\Psi = \{\log \psi_n\}_{n \in \N}$ is that $\Psi$ satisfies the bounded distortion \eqref{eq: bdd distortion general}. It is clear that subadditive potentials $\Psi$ considered in the previous subsection (i.e., $\Psi$ obtained from $\psi \in \D$ by \eqref{eq: psi_n from psi}, or equivalently, $\Psi$ satisfying \eqref{eq: psi for loc const}) has the bounded distortion property with $C=1$. 

%

\begin{rem}
For a subadditive potential $\Psi = \{\log \psi_n\}_{n \in \N}$ with bounded distortion, we can restate the Gibbs property \eqref{eq: gibbs for subadditive} and quasi-multiplicativity from Definition \ref{defn: qm for subadditive} by replacing $\psi(\I)$ to $\psi_n(x)$ for \textit{any} $x\in [\I]$.   

More precisely, an $f$-invariant measure $\mu \in \M(f)$ has the Gibbs property with respect to $\Psi = \{\log \psi_n\}_{n \in \N}$ if there exists $C \geq 1$ such that for any $n \in \N$ and $\I \in \L(n)$,
$$ 
C^{-1} \leq \frac{\mu(\I)}{e^{-n\P(\Psi)}\psi_n(x)}  \leq C
$$
for any $x \in [\I]$. This formulation resembles the Gibbs property of the additive setting \eqref{eq: gibbs ineq} more closely.

Quasi-multiplicativity of such sequence $\{\psi_n\}_{n \in \N}$ (or, equivalently, of the subadditive potential $\Psi$) is equivalent to the existence of $c>0$ and $k \in \N$ such that for any words $\I,\J \in \L$, there exists $\K = \K(\I,\J) \in \L$ with $|\K| \leq k$ such that
$$\psi_{|\I\K\J|}(x) \geq c \psi_{|\I|}(y) \psi_{|\J|}(z)$$
for any $x \in [\I\K\J]$, $y \in [\I]$, and $z \in [\J]$.
\end{rem}


The following proposition states that Proposition \ref{prop: unique eq as gibbs} remains valid for subadditive potentials with bounded distortion.

\begin{prop}
Let $\Psi = \{\log\psi_n\}_{n \in \N}$ be a subadditive potential with bounded distortion \eqref{eq: bdd distortion general}. If $\{\psi_n\}_{n \in \N}$ is quasi-multiplicative, then $\Psi$ has a unique equilibrium state. Such equilibrium state is ergodic and has the Gibbs property with respect to $\Psi$.
\end{prop}
\begin{proof}
Let $\psi \in \D$ be the submultiplicative function on $\L$ obtained from $\Psi$ as in \eqref{eq: simplification 2}:
$$\psi(\I) := \max\limits_{x\in [\I]}\psi_n(x)$$ 
Then, $\psi$ is quasi-multiplicative. Let $\widetilde{\Psi} = \{\log\widetilde{\psi}_n\}_{n \in \N}$ be the subadditive potential obtained from $\psi$ by \eqref{eq: psi_n from psi}. Note that $\widetilde{\Psi}$ satisfies \eqref{eq: psi for loc const}, and $\widetilde{\psi}_n$ and $\psi_n$ are related by the identity
$$\widetilde{\psi}_n(x) = \max\limits_{y \in [x]_n}\psi_n(y).$$
The proposition will follow from the following claim relating the thermodynamical objects of $\Psi$ and $\widetilde{\Psi}$.\\
\noindent\textbf{Claim}: $\P(\Psi) = \P(\widetilde{\Psi})$. Moreover, the set of equilibrium states of $\Psi$ is equal to the set of equilibrium states of $\widetilde{\Psi}$.
\begin{proof}[Proof of the claim]
Both statements made in the claim easily follow from the bounded distortion property on $\Psi$. 

For any $(n,\ep)$-separated set $E$, we have from the bounded distortion and the definition of $\widetilde{\psi}_n$ that
$$1 \leq \frac{\sum\limits_{x \in E}{\widetilde{\psi}_n}(x)}{\sum\limits_{x \in E}{\psi_n}(x)} \leq C.$$
Then, it follows from the definition of the subadditive pressure \eqref{eq: pressure} that $\P(\Psi) = \P(\widetilde{\Psi})$.

For the second statement in the claim, again from the bounded distortion, we have $\F(\Psi,\mu) = \F(\widetilde{\Psi},\mu)$ for any $f$-invariant measure $\mu$. Since the measure-theoretic entropy $h_{\mu}(f)$ does not depend on the potential, the second claim follows from the subadditive variational principle \eqref{eq: var prin}.
\end{proof}\noindent

Since $\widetilde{\Psi}$ satisfies \eqref{eq: psi for loc const}, we obtain the unique equilibrium state $\mu$ of $\widetilde{\Psi}$ from Proposition \ref{prop: unique eq as gibbs}. From the claim, we conclude that $\mu$ is the unique equilibrium state of $\Psi$. To conclude the proof, we note from the bounded distortion property that the Gibbs property of $\mu$ with respect to $\widetilde{\Psi}$ is equivalent to the Gibbs property of $\mu$ with respect to $\Psi$.
\end{proof}
 
Recalling that the singular value potential of a $\glr$-valued function $\A$ is defined by 
$$\Phi_\A^s:=\{\log\vps(\A^n(\cdot))\}_{n \in \N},$$ 
the following lemma shows that the singular value potentials $\Phi_\A^s$, $s \in [0,\infty)$ of a fiber-bunched $\A \in \C^\alpha_b(\Sig,\glr)$ have bounded distortion. 

\begin{lem}[bounded distortion]\label{lem: bdd distortion} Let $\A$ be a \hol and fiber-bunched $\glr$-valued function on $\Sig$. Then $\Phi_\A^s$ has bounded distortion for any $s \in [0,\infty)$.  
\end{lem}
\begin{proof}
From \hol continuity of the canonical holonomies \eqref{eq: hol holder}, we can fix $c>1$ such that $\|H_{x,y}^{s/u}\|$ is bounded above by $c$ whenever $d(x,y) \leq \theta$. Hence, for any $x,y \in \Sig$ with $d(x,y) \leq \theta$, we have that $\vps(H^{s/u}_{x,y})$ is uniformly bounded above by $c^s$. 

Consider any $n \in \N$, $\I \in \L(n)$, and $x,y\in [\I]$.
Then, setting $z:=[x,y]$ and using (2) of Definition \ref{defn: hol rel} as well as the submultiplicativity \eqref{eq: submult} of $\vps$, we have 
$$c^{-2s}\vps(\A^n(x)) \leq  \vps(\A^n(z)) =\vps( H^s_{f^nx,f^nz} \circ \A^n(x) \circ H^s_{z,x}) \leq c^{2s} \vps(\A^n(x)).$$
Using the canonical unstable holonomy instead, we have $c^{-2s} \leq \vps(\A^n(y))/\vps(A^n(z))\leq c^{2s}$.
Then, the statement follows by setting the constant $C$ equal to $c^{4s}$.
\end{proof}

\begin{rem}\label{rem: unif cts hol implies bdd dist}
Note that Lemma \ref{lem: bdd distortion} also holds for any $\A \colon \Sig \to \glr$ admitting uniformly continuous holonomies $H^{s/u}$. 

Moreover, the canonical holonomies $H^{s/u}$ vary continuously in $\A$. Hence, for a fixed $s$, by increasing $C$ from Lemma \ref{lem: bdd distortion} if necessary, the bounded distortion holds on $\Phi_\B^s$ for all $\B \in C_b^\alpha(\Sig,\glr)$ sufficiently close to $\A$ with the uniform constant $C$. 
\end{rem}

Recall that the subset $\U$ of $C^\alpha_b(\Sig,\glr)$ consists of typical $\glr$-valued functions.  
Using the uniform constant $c$ from Theorem \ref{thm: E}, we show that the subadditive pressure $\P(\Phi_\A^s)$ is continuous on $\U \times [0,\infty)$ by adapting the proof of Fekete's lemma. 
Since the equilibrium state of $\Phi^s_\A$ for a typical $\A \in \U$ is unique from quasi-multiplicativity, it follows that the unique equilibrium state also varies continuously on $\U \times [0,\infty)$.
\begin{thm*}[Theorem B]
~
\begin{enumerate}
\item The map $(\A,s) \mapsto \P(\Phi_\A^s)$ is continuous on $\U \times [0,\infty)$.
\item For each $\A \in \U$ and $s \in [0,\infty)$, the singular value potential $\Phi_\A^s$ has a unique equilibrium state $\mu_{\A,s}$, which also varies continuously on $\U \times [0,\infty)$.
\end{enumerate}
\end{thm*}

The proof of Theorem \ref{thm: B} appears in Section 5.
From the definition and the subadditivity, the map $(\A,s) \mapsto \P(\Phi_\A^s)$ is upper semi-continuous,
 and hence is generically continuous on its domain $C^\alpha(\Sig,\glr)$. Theorem \ref{thm: B} establishes that, when restricted to $C^\alpha_b(\Sig,\glr)$, the subadditive pressure varies continuously on an open and dense subset $\U$. 
 
Cao, Pesin, and Zhao \cite{cao2018dimension} recently showed that the map $(\A,s)\mapsto \P(\Phi_\A^s)$ is jointly continuous on $C^\alpha(\Sig,\glr) \times [0,\infty)$, and Theorem \ref{thm: B} (1) is implied by their result. However, the methods of proof are different. Cao, Pesin, and Zhao construct a horseshoe with dominated splitting which carries most of the pressure. Using the structural stability of horseshoes, they establish the lower semi-continuity of the pressure. See \cite{cao2018dimension} for details.
On the other hand, we compare $\P(\Phi_\A^s)$ to $\P(\Phi_\B^s)$ for $\B$ sufficiently close to $\A$ using uniform constants from simultaneous quasi-multiplicativity of Theorem \ref{thm: E}.

For similar results in this direction, Feng and Shmerkin \cite{feng2014non} showed that locally constant functions are continuity points of $\P(\Phi_\A^s)$ in $L^\infty(\Sig,M_{d \times d}(\R))$.

\subsection{Exterior algebra}\label{subsection: exterior algebra}
We will make use of the exterior algebra in studying the singular value potential $\Phi_\A^s$. For $1 \leq k \leq d$, we denote the $k$-th exterior power of $\R^d$ by $(\R^d)^{\wedge k}$. It is a ${d\choose k}$-dimensional $\R$-vector space spanned by decomposable vectors $v_1 \wedge \ldots \wedge v_k$ with the usual identifications.

Any linear transformation $A$ and the standard inner product $\langle \cdot , \cdot \rangle$ on $\R^d$ naturally extend to $(\R^d)^{\wedge k}$: for any two decomposable vectors $v_1 \wedge \ldots \wedge v_k, u_1 \wedge \ldots \wedge u_k \in (\R^d)^{\wedge k}$, we have
\begin{align*}
A^{\wedge k}(v_1 \wedge \ldots \wedge v_k) &:= Av_1 \wedge \ldots \wedge Av_k,\\
\langle v_1 \wedge \ldots \wedge v_k, u_1 \wedge \ldots \wedge u_k \rangle &:= \det (\langle v_i,u_j\rangle)_{1 \leq i,j \leq k},
\end{align*}
and we extend it to the entire $(\R^d)^{\wedge k}$ by linearity. The exterior algebra satisfies the following properties: for any linear transformations $A,B$ of $\R^d$, 
\begin{align*}
(AB)^{\wedge k} = A^{\wedge k }B^{\wedge^k}~,~  (A^{\wedge k})^\intercal = (A^\intercal)^{\wedge k},\\
\|A^{\wedge k }\| = \alpha_1(A)\ldots \alpha_k(A) = \vphi^k(A).
\end{align*}
Under the induced inner product on $(\R^d)^{\wedge k}$, it follows that $\Phi_\A^k = \Phi^1_{\A^{\wedge k}}$. 


\section{Quasi-multiplicativity}
In this section, we prove Theorem \ref{thm: E}. We will first illustrate the ideas by proving a simpler result, Theorem \ref{thm: A}. Building on the proof Theorem \ref{thm: A} and using an inductive argument, we will prove a more general result which we describe now.

In what follows, we let $\V_t, ~t=1,2,\ldots,\kappa $ be normed $\R$-vector spaces of dimension $d_t$. 
For any $\A_t\colon \Sig \to \text{GL}(\V_t)$, we define $\widetilde{\vphi}_{\A_t}^1 \colon \L \to \R^+_0$ analogously to \eqref{eq: A(I)} and \eqref{eq: simplification 1}: $$\widetilde{\vphi}_{\A_t}^1(\I) = \|\A_t(\I)\| := \max\limits_{x \in [\I]}\|\A_t^{|\I|}(x)\|.$$

\begin{thm}\label{thm: qm general} 
Let $\A_t \colon \Sig \to \text{GL}(\V_t),~ t =1,2,\ldots,\kappa$ be \hol functions admitting uniformly continuous stable and unstable holonomies. Suppose there exist a fixed point $p \in \Sig$ and a homoclinic point $z \in \W^s(p) \cap \W^u(p) \setminus \{p\}$ such that each $\A_t$ satisfies the pinching (A0) and the twisting (B0) conditions of Definition \ref{defn: 1-typical} at $p$ and $z$. Then the singular value functions $\widetilde{\vphi}^1_{\A_t},~t= 1,2,\ldots, \kappa$ are simultaneously quasi-multiplicative: there exist $c>0$, $k \in \N$ such that for any words $\I,\J \in \L$, there exists $\K = \K(\I,\J) \in \L$ with $|\K| \leq k$ such that $\I\K\J \in \L$ and that for each $1 \leq t\leq \kappa$, we have 
$$\|\A_t(\I\K\J)\| \geq c\|\A_t(\I)\|\cdot \|\A_t(\J)\|.$$
Moreover, the constants $c,k$ can be chosen uniformly in a small neighborhood of each $\A_t$.
\end{thm}


\begin{rem} 
The first statement is the main content of Theorem \ref{thm: qm general}; the uniform choice of the constants $c$ and $k$ follows from the fact that all parameters vary continuously on the data $\A_t$.

Although the constants $c,k$ can be chosen uniformly in a small neighborhood of each $\A_t$, we cannot necessarily choose the connecting word $\K$ uniformly. See Remark \ref{rem: nonuniform K0}.
\end{rem}

We will prove Theorem \ref{thm: qm general} in subsection \ref{subsection: pf qm general}. In subsection \ref{subsection: pf thm E}, we will then show that Theorem \ref{thm: E} follows as a corollary of Theorem \ref{thm: qm general}.

\subsection{Preliminary linear algebra}
We first collect preliminary lemmas and relevant constants needed in the proof of Theorem \ref{thm: A} and Theorem \ref{thm: qm general}. Throughout the section, $\V$ is a finite dimensional $\R$-vector space equipped with a norm $\|\cdot\|$.

\begin{defn} For $A \in \text{End}(\V)$, we choose a singular value decomposition (SVD)
$$A = U\Lambda V^\intercal,$$
where the singular values in $\Lambda$ are listed in a non-increasing order.
We define $u(A)$ and $v(A)$ as the first column of $U$ and $V$, respectively. 

If the singular values of $A$ are distinct, then the SVD of $A$ is unique (up to signs), and hence so are $u(A)$ and $v(A)$. If there are repeated singular values, then the singular value decomposition of $A$ is not necessarily unique. In this case, we simply choose a singular value decomposition of $A$, and set $u(A)$ and $v(A)$ accordingly. 
\end{defn}

Roughly speaking, $u(A)$ and $v(A)$ can be thought of as the most expanding direction of $A^*=A^\intercal$ and $A$, respectively.
From the definition, we have 
\begin{equation}\label{eq: linear algebra}
\|A\|u(A) = Av(A).
\end{equation}
Throughout the section, when we measure an angle between nonzero vectors, we mean the angle between the lines spanned by the vectors. Similarly, when we measure an angle between a nonzero vector $v$ and a hyperplane $\WW$, we mean the minimum angle $ \mangle (v,w)$ over all $w \in \WW\setminus \{0\}$.  Also, we will not distinguish between a vector in $\V \setminus \{0\}$ and its corresponding point in the projective space $\mathbb{P}(\V)$ when there is no confusion.
We have an easy lemma from linear algebra.
\begin{lem}\label{lem: linear algebra 1}
Given any $A \in \text{Aut}(\V)$ and any $w \in \V$, we have
$$\|Aw\| \geq \cos \mangle \big( w,v(A)\big) \|A\| \cdot \|w\|.$$
\end{lem}

\begin{proof}
Let $v = v(A)$, and write $w = av+v'$ where $|a| =\|w\| \cos\mangle(w,v)$ and $v' \in v^\perp$. Letting $u = u(A)$, we have from \eqref{eq: linear algebra} that
$$Aw = a\|A\|u+Av'.$$
Since the singular vectors are pairwise orthogonal (i.e., columns of $U$ are pairwise orthogonal), we have $Av' \in u^\perp$ and the lemma follows.
\end{proof}


Recall that the \textit{co-norm} $m(A)$ of $A \in \text{GL}(\V)$ is defined by $$m(A) = \|A^{-1}\|^{-1}.$$
The following lemma will be useful in proving Theorem \ref{thm: qm general}. 

\begin{lem}\label{lem: linear algebra 2}
Let $\theta>0$ be given and $A,B,C,D \in \text{Aut}(\V)$ such that
$$\mangle \big( B^*v(A), (Cu(D))^\perp \big) >\theta.$$
Then,
$$\|ABCD\| \geq \|A\|\cdot \|D\|\cdot \sin(\theta) \frac{m(B)^2m(C)^2}{\|B\|\|C\|} .$$
\end{lem}
\begin{proof}
We have
\begin{align*}
\|BCu(D)\|\cos \mangle \big( BCu(D), v(A)\big)&=\langle v(A),BCu(D)\rangle,\\ 
&= \langle B^*v(A), Cu(D) \rangle,\\ 
&\geq \|B^*v(A)\|\|Cu(D)\|\sin(\theta).
\end{align*}
Hence, $$\cos \mangle \big( BCu(D), v(A)\big)\geq \sin(\theta) \frac{m(B)m(C)}{\|B\|\|C\|}.$$
It then follows from \eqref{eq: linear algebra} and Lemma \ref{lem: linear algebra 1} that 
\begin{align*}
\|ABCD\| \geq \|ABCDv(D)\| &= \|D\|\cdot \|ABCu(D)\|,\\
&\geq  \|D\|\cdot \cos \mangle \big( BCu(D), v(A)\big) \|A\|\cdot \|BCu(D)\|,\\
&\geq \|A\|\cdot \|D\|\cdot \sin(\theta) \frac{m(B)m(C)}{\|B\|\|C\|} \cdot m(BC).
\end{align*}
This completes the proof.
\end{proof}
%

We will also make use of the adjoint cocycle. For a cocycle $F_\A$ generated by a $\text{GL}(\V)$-valued function $\A$ over $f$, we define the \textit{adjoint cocycle} $F_{\A}^*$ over $f^{-1}$ generated by $\A_*$ where $\A_*$ is defined by the relation
\begin{equation}\label{def: adjoint defn}
\langle\A_*(x)u, v \rangle = \langle u , \A(f^{-1}x)v \rangle ~~\text{ for all } x\in \Sig \text{ and } u,v \in \V.
\end{equation}

Suppose $F_\A$ admits holonomies $H^{s/u}$. Then the adjoint cocycle $F_\A^*$ also admits holonomies given by $$H^{s,*}_{x,y} = (H^{u}_{y,x})^*~\text{ and }~H^{u,*}_{x,y} = (H^s_{y,x})^*.$$
This can be easily seen by plugging $u = (H^s_{x,y})^*\widetilde{u}$ and $v = H^s_{f^{-1}y,f^{-1}x} \widetilde{v}$ into \eqref{def: adjoint defn} for some $y$ in the stable set of $x$ with respect to $f$. 
The following lemma shows that many properties of $\A$ carry over to $\A_*$. 
\begin{lem}\label{lem: adjoint typical}
Let $\A \in C^\alpha(\Sig,\glr)$. Then, 
\begin{enumerate} 
\item $F_\A^*$ is fiber-bunched if and only if $F_\A$ is fiber-bunched.
\item $\A_*$ is 1-typical if and only if $\A$ is 1-typical. 
\item $\A_*$ is typical if and only if $\A$ is typical.
\end{enumerate}
\end{lem}

\begin{proof}
See Lemma 7.2 of \cite{bonatti2004lyapunov} for the proof of (1). The setting in \cite{bonatti2004lyapunov} is $\slr$-valued cocycles, but the proof readily extends to $\glr$-valued cocycles.
For (2), we note that the eigenvalues of the adjoint matrix $P^*$ are equal to the eigenvalues of $P$; in particular, they are simple and distinct in modulus. Indeed, if we define $\WW_j$ to be the hyperplane spanned by all but the $j$-th eigenvector $v_j$ of $P$, then the $j$-th eigendirection of $P^*$ is given by $w_j :=(\WW_j)^\perp$: for any $1 \leq i \neq j \leq d$, we have
$$
\langle v_i, P^*w_j\rangle = \langle Pv_i,w_j \rangle= \lambda_i \langle v_i, w_j \rangle = 0.
$$
The twisting condition (B0) from Definition \ref{defn: 1-typical} is then equivalent to 
$$\langle \psi_p^z(v_i),w_j\rangle \neq 0 \text{ for all }1 \leq i,j \leq d.$$
Hence, $\langle v_i ,(\psi_p^z)^*w_j \rangle \neq 0$ for all $1 \leq i,j \leq d$; this is equivalent to $\A_*$ being 1-typical because $\psi_p^{z,*} =( \psi_p^z)^*$.
(3) then trivially follows from (2).
\end{proof}

For $v \in \mathbb{P}(\V)$, let the \textit{cone around $v$ of size $\ep$} be 
$$\CC(v,\ep):=\{w \in \mathbb{P}(\V) \colon \mangle (v,w ) < \ep\}.$$

If $P \in \text{GL}(\V)$ has simple eigenvalues of distinct norms, then any $v \in \mathbb{P}(\V)$ can be mapped close to one of the eigendirections of $P$ under iterations of $P$. Even though the number of iterations needed depends on the given direction $v$, the following lemma shows that such number of iterations can be uniformly bounded above, independent of $v$. A quick illustration of ideas in $\mathbb{P}(\R^3)$ is as follows: suppose $\{v_i\}_{1\leq i \leq 3}$ are eigendirections of $P$ with $|\lambda_1|>|\lambda_2|>|\lambda_3|$. If given $v$ is already close to some $v_i$, then no iteration of $P$ is necessary. If not, then a large but bounded number of iterations of $P$ will either map $v$ close to one of the $v_i$'s or map it out of the $\ep$-neighborhood of span$\{v_2,v_3\}$ for some fixed $\ep>0$, in which case further bounded number of iterations of $P$ will map it close to $v_1$.

\begin{lem}\label{lem: turn}
Suppose $\V$ is $d$-dimensional, and $P \in \text{GL}(\V)$ has simple eigenvalues of distinct norms with corresponding eigenvectors $\{v_i\}_{1 \leq i \leq d}$. Given $\ep>0$, there exists $N = N(\ep)\in \N$ such that for any $v \in \mathbb{P}(\V)$, there exists $n = n(v) \leq N$ such that 
$$P^nv \in \bigcup\limits_{i=1}^d\CC(v_i,\ep).$$ 
\end{lem}
\begin{proof}
Without loss of generality, suppose that the eigenvalues $\{\lambda_i\}_{1 \leq i \leq d}$ of $P$ corresponding to $\{v_i\}_{1\leq i \leq d}$ are decreasing in modulus. We adopt the same notation as in Lemma \ref{lem: adjoint typical} and define $\WW_j$ to be the hyperplane spanned by all but the $j$-th eigenvector $v_j$. 
We make a few observations:
\begin{enumerate}
\item 
If $v$ is close to each hyperplane $\WW_i$ for every $i \neq j$, then $v$ has to be close to $v_j$. 
We fix $\eta>0$ depending on the given $\ep>0$ such that the following holds: if $v \in \mathbb{P}(\V)$ with $\mangle(v,\WW_i)<\eta$ for all $i \in \{1,\ldots, j-1,j+1,\ldots,d\}$, then $v \in  \CC(v_j,\ep)$.

\item Let $v = \sum\limits_{i=1}^d c_iv_i$. If the angle $\mangle(v,W_j)$ is not too small, then the ratio $|c_j/c_i|$ is not too small (if $c_i=0$, the ratio is $\infty$) for every $i$. Since $|\lambda_j| > |\lambda_{i}|$ for all $i \geq j+1$, we choose some large $m\in \N$ such that $|\lambda_j^mc_j/\lambda_i^mc_i|$ is sufficiently large for all $i \geq j+1$; this implies that $P^mv$ makes a small angle with each $\WW_i$ for $i \geq j+1$.

Formally, for $\eta>0$ chosen in (1), we choose $m\in \N$ such that for any $1\leq j \leq d$, if $\mangle(v, \WW_j)>\eta$, then the angle $P^mv$ makes with each hyperplane $\WW_i$ with $ i\geq j+1$ is at most $\eta$. 
The existence of such $m\in \N$ follows from the simplicity of the eigenvalues of $P$. 

\end{enumerate}
We claim that for any $v \in \mathbb{P}(\V)$, there exists $0 \leq k \leq d-1$ with $P^{km}v \in \bigcup\limits_{i=1}^d\CC(v_i,\ep)$. 

If $w_0:=v \in \bigcup\limits_{i=1}^d\CC(v_i,\ep)$, then there is nothing to be done; we set $k=0$. 

If $w_0 \not\in \bigcup\limits_{i=1}^d\CC(v_i,\ep)$, then we find the smallest $j_0 \in \N$ such that $\mangle(w_0,\WW_{j_0})>\eta$. From the choice of $m$, the angle $w_1:=P^mw_0$ makes with each hyperplane $\WW_i$, $i = j_0+1,\ldots,d$ is smaller than $\eta$. 

If $w_1 \in \bigcup\limits_{i=1}^d\CC(v_i,\ep)$, then we set $k = m$. If not, from $w_1 \not\in \CC(v_{j_0},\ep)$ and (1), there exists some $i \in \{1,\ldots, j_0 - 1 , j_0+1,\ldots,d\}$ such that $\mangle(w_1,\WW_i)>\eta$. Since we already know $w_1$ makes an angle less than $\eta$ with each $\WW_i$ with $i\geq j_0+1$, such $i$ is necessarily smaller than $j_0$. We then set $j_1$ to be the smallest number (necessarily smaller than $j_0$) among such $i$. Again from the choice of $m$, the angle $w_2:=P^mw_1$ makes with each $\WW_i$, $i=j_1+1,\ldots,d$ is smaller than $\eta$.

We repeat the process inductively as follows: given $w_n:=P^mw_{n-1}$ from the previous step, we set $k = nm$ if $w_n \in \bigcup\limits_{i=1}^d\CC(v_i,\ep)$. If not, from $w_n \not\in \CC(v_{j_{n-1}},\ep)$ and (1), we can necessarily find some $i$ smaller than $j_{n-1}$ such that $\mangle(w_n,\WW_i)>\eta$. We set $j_n$ to be the smallest such $i$. Then, $w_{n+1}:=P^mw_{n}$ makes an angle less than $\eta$ with $\WW_i$, $i =j_n+1,\ldots,d$.

We continue this process until $j_n = 1$. From the construction, $\mangle(w_{n+1},\WW_i)<\eta$ for all $i=2,\ldots,d$, which implies that $w_{n+1} \in \CC(v_1,\ep)$. Note that $j_0 \leq d-1$, because if $j_0$ were equal to $d$, then $w_0$ must have been in $\CC(v_d,\ep)$ contradicting $w_0 \not\in \bigcup\limits_{i=1}^d\CC(v_i,\ep)$. Since $\{j_n\}$ is a strictly decreasing sequence bounded below by 1, the inductive process necessarily terminates in at most $d-1$ steps. We complete the proof by setting $N:=(d-1)m$.
\end{proof}

\begin{rem}\label{rem: N uniform}
Since the eigenvalues of $P$ from Lemma \ref{lem: turn} vary continuously in $P$, we can choose $N$ to work uniformly near $P$: given $\ep>0$, there exists $N\in \N$ such that for any $\widetilde{P} \in \text{GL}(\V)$ sufficiently close to $P$ and any $v \in \mathbb{P}(\V)$, there exists $n  = n(v,\widetilde{P}) \leq N$ such that $\widetilde{P}^nv \in \bigcup\limits_{i=1}^d\CC(\widetilde{v}_i,\ep)$ where $\{\widetilde{v}_i\}_{1\leq i \leq d}$ are distinct eigendirections of $\widetilde{P}$. 
\end{rem}

In the following lemma, we also adopt the same notations from Lemma \ref{lem: adjoint typical}.

\begin{lem}\label{lem: lin algebra immediate}
Let $\ep>0$ be given, and suppose $P,\psi,R \in \text{GL}(\V)$ and $\ell \in \N$ satisfy the following properties:
\begin{itemize}
\item $P$ has simple real eigenvalues of distinct norms, 
\item For any $v \in \bigcup\limits_{i=1}^d\CC(v_i,\ep)$, we have $\mangle(\psi (v),\WW_i) >\ep$ for each $i$,
\item $\mangle(R(v),v) <\ep/2$ for any $v\in \mathbb{P}(\V)$,
\item For any $v\in \mathbb{P}(\V)$ with $\mangle(v,\WW_i)>\ep$ for each $i$, we have $P^\ell v \in \CC(v_1,\ep/2)$.
\end{itemize}
Then for any $v \in \bigcup\limits_{i=1}^d\CC(v_i,\ep/2)$, we have
$$P^\ell \psi R (v) \in \CC(v_1,\ep/2).$$
\end{lem}
\begin{proof}
The proof is immediate from the properties of $P,\psi,R$ and $\ell$.
\end{proof}

\subsection{Proof of Theorem \ref{thm: A}}\label{subsection: pf thm A}
In this subsection, we prove Theorem \ref{thm: A}.

\begin{thm*}[Theorem \ref{thm: A}]
Every $\A \in \U$ is quasi-multiplicative. Moreover, the constants $c,k$ in \eqref{eq: qm intro} can be chosen uniformly in a neighborhood of $\A$ in $\U$. 
\end{thm*}

\begin{proof}[Proof of Theorem \ref{thm: A}]
Given $\A \in \U$, we will set uniform constants $c$ and $k$ such that for any given $\I,\J\in \L$, there exists $\K \in \L$ with $|\K| \leq k$ such that quasi-multiplicativity \eqref{eq: qm intro} holds.

Let $p$ and $z$ be the periodic and homoclinic point given by the hypothesis. 
For simplicity, we assume that $p$ is a fixed point of $f$. In the case where the reference point $p$ is a periodic point, we replace $f$ by its suitable power so that $p$ becomes a fixed point and the proof readily extends with relevant modifications. 
From Remark \ref{rem: homoclinic pts}, we also assume that $z$ is on $\Wloc^u(p)$.
\\\\
\textbf{Step 1.} We begin by setting up the notations and constants to be used in the proof.

\begin{itemize}
\item For any $(\omega,n) \in \Sig \times \N$, we identify it with the orbit segment starting at $\omega$ of length $n$.

\item Let $\{v_i\}_{1 \leq i \leq d}$ be eigendirections of $P = \A(p)$ listed in the order of decreasing modulus of eigenvalues. Similarly, we denote the eigendirections of $P_*:=\A_*(p)$ by$\{w_i\}_{1\leq i \leq d}$.
We define $\WW_j$ be the hyperplane in $\R^d$ spanned by all $v_i$'s except $v_j$. As in the proof of Lemma \ref{lem: adjoint typical}, we have $w_j=( \WW_j)^\perp$ for each $1\leq j \leq d$.

\item The angle formed by the top eigendirections $v_1$ and $w_1$ of $P$ and $P_*$ is necessarily bounded away from $\pi/2$.
Let
$$\beta := \mangle(v_1,w_1^\perp) = \mangle(v_1,\WW_1)>0.$$

\item The twisting condition (B0) implies that there exists $\ep_0>0$ such that
\begin{equation}\label{eq: uniform twist 0}
\mangle\big(\psi_p^z(v),\WW_j \big)>\ep_0,
\end{equation}
for all $1 \leq j \leq d$ whenever $v \in \bigcup\limits_{i=1}^{d} \CC(v_i,\ep_0)$. Fix such an $\ep_0 \in (0,\beta/8)$.

\item Suppose $a,b,c,d\in \Sig$ are related by
$$[a,c] = b ~~\text{ and }~~[c,a] = d,$$
where the bracket operation is defined in \eqref{eq: bracket}.
Then we say such points form a \textit{rectangle} with vertices $a,b,c,$ and $d$, and denote it by $[a,b,c,d]_r$.

\begin{figure}[H]
\caption{}
\centering
\includegraphics[width=0.25\columnwidth]{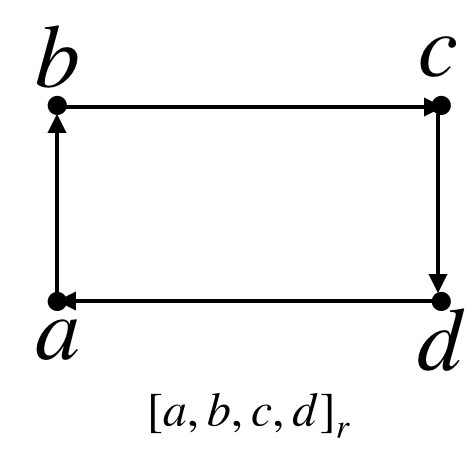}
\label{fig: figure_6}
\end{figure}

Note that a rectangle is defined by prescribing two opposite vertices. All rectangles appearing in the proof will have one of its vertices at $p$. 
\item 
For $q \in \Sig$ in the local neighborhood of $p$, but not on $\Wloc^s(p) \cup\Wloc^u(p)$, consider the rectangle $[p,x,q,y]_r$ having $p$ and $q$ as opposite vertices.
We define ``the holonomy of the rectangle $[p,x,q,y]_r$'' by
\begin{equation}\label{eq: hol rectangle 0}
R_q : = H^u_{y,p}\circ H^s_{q,y}\circ H^u_{x,q}\circ H^s_{p,x}.
\end{equation}

Since canonical holonomies are uniformly continuous, the holonomy composition $R_q$ uniformly approaches the identity as the rectangle degenerates (i.e., as a pair of opposite sides degenerates to a point) to a line or a point. 

\item 
Recall $\theta \in (0,1)$ is the hyperbolicity constant defining the metric on the base $\Sig$. We fix $m\in \N$ such that the following conditions hold: suppose $[p,x,q,y]_r$ is a rectangle.
\begin{enumerate}[label=(\roman*)]
\item If $[p,x,q,y]_r$ has an edge whose length is at most $\theta^m$, then
$$\mangle(R_q(v),v) < \frac{\ep_0}{2} \text{ for any }v \in \mathbb{P}(\R^d).$$
\item If all edges of $[p,x,q,y]_r$ are no longer than $\theta^m$, then
$$\mangle\big(H^{u}_{x,q}\circ H^{s}_{p,x}(v),v \big)< \ep_0/2~~ \text{ and }~~\mangle\big(H^{s}_{y,q}\circ H^{u}_{p,y}(v),v \big) < \ep_0/2,$$
for any $v \in \mathbb{P}(\R^d)$.
\end{enumerate}
The existence of such $m\in \N$ is guaranteed from the uniform continuity of the canonical holonomies $H^{s/u}$. 

\item 
Recall that we assumed $z\in \Wloc^u(p)$. Fix $\ell \in \N$ such that $f^\ell z \in \Wloc^s(p)$. Increase $\ell$ if necessary such that for any $v \in\bigcup\limits_{i=1}^{d} \CC(v_i,\ep_0)$, we have
$$P^\ell\psi_p^z(v) \in \CC(v_1,\ep_0/2).$$
The existence of such $\ell$ is guaranteed from \eqref{eq: uniform twist 0} and pinching condition (A0) on $P$.

Notice that further increasing $\ell$ does not disturb the defining properties of $\ell$. So, we further increase $\ell$ if necessary so that $d(f^\ell z,p) \leq \theta^{m}$.
 
\item Set $\Upsilon :=\max\left(\max\limits_{x \in \Sig}\|\A(x)\|,1\right)$ and $\varrho:=\min\left(\min\limits_{x\in \Sig}m(\A(x)),1\right)$. 
\item Using the uniform continuity of the canonical holonomies, we fix $C_0>1$ so that $\|H^{s/u}_{x,y}\| \leq C_0$ for any $x,y\in \Sig$ with $d(x,y) \leq \theta$. Increase $C_0$ if necessary so that it also serves as a constant for the bounded distortion property \eqref{eq: bdd distortion general} of the singular value potential $\Phi_\A^1$:
for any $n \in \N$ and $\I \in \L(n)$, we have
$$C_0^{-1} \leq \frac{\|\A^n(x)\|}{\|\A^n(y)\|} \leq C_0,$$
for any $x,y\in \I$.

\item Let $N\in \N$ be given by applying Lemma \ref{lem: turn} to $P$ and $\ep_0/2$. Then for any $v \in \mathbb{P}(\R^d)$, there exists $n=n(v) \leq N$ such that $P^nv \in \bigcup\limits_{i=1}^d \CC(v_i,\ep_0/2)$.
\item Let $k_1 := N+\ell$.
\end{itemize}

By adjusting the constants $\beta,\ep_0,m,\ell,\Upsilon,\varrho,C_0,N$ and $k_1$ in the order they are defined, we may assume that they work for the adjoint cocycle as well. For the adjoint cocycle, we interchange the role of $z$ and $f^\ell z$, and denote the corresponding points (on the orbit of $z$) by $\hat{z} \in \Wloc^u(p)$ and $f^{-\ell} \hat{z}\in \Wloc^s(p)$.

The constants $\beta,\ep_0,m,r,\ell,\Upsilon,\varrho,C_0,N$ and $k_1$ also work uniformly in a small neighborhood of $\A$. We will comment regarding the uniform choice of the constants $c,k$ at the end of the proof. \\

\noindent\textbf{Step 2.} 
Since the adjacency matrix $T$ is primitive, there exists $\bar{\tau} \in \N$ such that $T^{\bar{\tau}}>0$. Such $\bar{\tau}$ is the mixing rate of the system $(\Sig,f)$. Then for any given $\I\in \L$, there exists $\bar{\omega}_0 \in [\I] \cap \W^s(p)$ such that $f^{|\I|+\bar{\tau}}\bar{\omega}_0 \in \Wloc^s(p)$.

We set 
$$\omega_0:= f^\tau \bar{\omega}_0 \text{ where } \tau  = \tau(\I) := |\I|+\bar{\tau}+m.$$
Since $f^{|\I|+\bar{\tau}}\bar{\omega}_0$ is already on the local stable set $\Wloc^s(p)$ of $p$, we have $d(\omega_0,p) \leq \theta^m$.

Let $$u_{\bar{\omega}_0}:=H^s_{\omega_0,p}u(\A^\tau(\bar{\omega}_0)).$$ Lemma \ref{lem: turn} implies that there exists $n =n(u_{\bar{\omega}_0}) \leq N$ such that $P^n u_{\bar{\omega}_0} \in \CC(v_i,\ep_0/2)$, for some $1 \leq i \leq d$. From $u_{\bar{\omega}_0}$ and $n$, we construct a new point 
$$\bar{\omega}_{\I} = f^{-\tau-n}[z,f^n \omega_0];$$
note that $\bar{\omega}_\I \in \Wloc^u(\bar{\omega}_0) \cap [\I]$.
We set 
$$\omega_\I := f^\tau \bar{\omega}_\I, \text{ and }\widetilde{\omega}_\I:= f^{n+\ell} \omega_\I.$$ 
The forward orbit segment starting at $\bar{\omega}_\I \in [\I]$ first comes close to $p$, arriving at $\omega_\I$, then dwells near $p$ for $n$ iterates, and then shadows the orbit segment from $z$ to $f^\ell z$ to finally land on $\Wloc^{s}(p)$ at the point $\widetilde{\omega}_\I$. 
Since $n$ is bounded above by $N$, the length of the orbit segment $(\omega_\I,n+\ell)$ is bounded above by $k_1$. 

The holonomy of the rectangle with opposite vertices at $p$ and $f^n\omega_\I = f^{-\ell}\widetilde{\omega}_\I$ is given by $$R_{f^{-\ell}\widetilde{\omega}_\I} = H^u_{z,p}H^s_{f^{-\ell} \widetilde{\omega}_\I,z}H^u_{f^n\omega_0,f^{-\ell} \widetilde{\omega}_\I}H^s_{p,f^n\omega_0}.$$
Combining this with the relation $H^s_{\widetilde{\omega}_\I,f^\ell z}\A^\ell(f^{-\ell}\widetilde{\omega}_\I) = \A^{\ell}(z)H^s_{f^{-\ell} \widetilde{\omega}_\I,z}$ and \eqref{eq: psi_p^z}, we obtain 
\begin{align}\label{eq: using psi_p^z}
\begin{split}
H^s_{\widetilde{\omega}_{\I},p}\A^{n+\ell}(\omega_{\I})H^u_{\omega_0,\omega_{\I}}
&=H^s_{\widetilde{\omega}_{\I},p}\A^\ell(f^{-\ell}\widetilde{\omega}_\I )\A^n(\omega_\I) H^u_{\omega_0,\omega_\I} \\
&=H^s_{\widetilde{\omega}_{\I},p}H^s_{f^\ell z,\widetilde{\omega}_\I}\A^{\ell}(z) H^s_{f^{-\ell} \widetilde{\omega}_\I,z}H^u_{f^n\omega_0,f^{-\ell} \widetilde{\omega}_\I} \A^n(\omega_0)\\
&= H^s_{f^\ell z,p}\A^{\ell}(z) H^u_{p,z}R_{f^{-\ell}\widetilde{\omega}_{\I}}H^s_{f^{n}\omega_{0},p} \A^{n}(\omega_0)\\
&=P^\ell\psi^{z}_{p} R_{f^{-\ell}\widetilde{\omega}_{\I}} H^s_{f^{n}\omega_{0},p} \A^{n}(\omega_0)\\
&=P^\ell\psi^{z}_{p} R_{f^{-\ell}\widetilde{\omega}_{\I}} P^{n} H^s_{\omega_0, p}.
\end{split}
\end{align}

Then $u_{\bar{\omega}_\I}:=H^s_{\widetilde{\omega}_\I,p}\A^{n+\ell}(\omega_\I)H^u_{\omega_0,\omega_\I}u(\A^\tau(\bar{\omega}_0))$ is related to $u_{\bar{\omega}_0}$ as follows: 
\begin{align}
\begin{split}\label{eq: ut relation 0}
u_{\bar{\omega}_\I}&=H^s_{\widetilde{\omega}_\I,p}\A^{n+\ell}(\omega_\I)H^u_{\omega_0,\omega_\I}u(\A^\tau(\bar{\omega}_0))\\
&=P^\ell\psi^{z}_{p} R_{f^{-\ell}\widetilde{\omega}_{\I}} P^{n}H^s_{\omega_0,p} u(\A^\tau(\bar{\omega}_0))\\
&=P^\ell\psi^{z}_{p} R_{f^{-\ell}\widetilde{\omega}_{\I}} P^{n}u_{\bar{\omega}_0}.
\end{split}
\end{align}
From \eqref{eq: ut relation 0}, it follows that 
\begin{equation}\label{eq: close to v1}
u_{\bar{\omega}_\I} \in \CC(v_1,\ep_0/2).
\end{equation}
Indeed, the choice of $n = n(u_{\bar{\omega}_0})$ gives $P^nu_{\bar{\omega}_0} \in \CC(v_i,\ep_0/2)$ for some $1 \leq i \leq d$. Since the edge between $p$ and $f^n\omega_0$ is no longer than $\theta^m$, $R_{f^{-\ell} \widetilde{\omega}_\I}$ doesn't move any line off itself more than $\ep_0/2$ in angle. 
Lemma \ref{lem: lin algebra immediate} then gives \eqref{eq: close to v1}.
Note from the choice of $\ell$, we have $d(\widetilde{\omega}_\I,p) \leq \theta^m$. This fact will be used in Step 4.

Let us briefly summarize what we have done so far. From a given word $\I\in \L$, we construct an orbit segment  $(\bar{\omega}_0,\tau)$ starting at $\bar{\omega}_0 \in [\I]$ and ending at $\omega_0 \in \Wloc^s(p)$ using the mixing property of the base system $(\Sig,f)$. We do not however have any control of the singular direction $u_{\bar{\omega}_0}$; it could be anywhere in $\mathbb{P}(\R^d)$. So we construct a new orbit segment $(\bar{\omega}_\I,\tau+n+\ell)$ which first shadows the orbit of $\bar{\omega}_0$ for time $\tau+n$ and then shadows the orbit of $z$ for time $\ell$. By choosing $n$ in such a way that $P^n u_{\bar{\omega}_0}$ is close to one of the eigendirections of $P$, we ensure that $u_{\bar{\omega}_\I}$ is close enough to the top eigendirection $v_1$ of $P$.
\\\\
\textbf{Step 3.}
We apply the argument in Step 2 to the adjoint cocycle $\A_*$ with $\hat{z}$ and $f^{-\ell}\hat{z}$ playing the role of $z$ and $f^\ell z$.

Similar to $\bar{\omega}_0$, we obtain $\hat{\iota}_0 \in f^{|\J|}\J$ from the mixing property of $(\Sig,f)$ such that $$\iota_0:=f^{-\tau(\J)}\hat{\iota}_0 \in \Wloc^u(p) \text{ where }\tau(\J) = |\J|+\bar{\tau}+m.$$
Applying Lemma \ref{lem: turn} to $P_*$ and the direction $H^{s,*}_{\iota_0,p}u(\A_*^{\tau(\J)}(\hat{\iota}_0))$ gives $\hat{n} \leq N$ such that $P_*^{\hat{n}}H^{s,*}_{\iota_0,p}u(\A_*^{\tau(\J)}(\hat{\iota}_0))$ belongs to the cone $\CC(w_i,\ep_0/2)$ for some $1\leq i\leq d$. Define $$\hat{\iota}_\J:=f^{\tau(\J)+\hat{n}}[f^{-\hat{n}}\iota_0,\hat{z}],$$ and set   
 $$\iota_\J = f^{-\tau(\J)}\hat{\iota}_{\J} \text{ and } \widetilde{\iota}_\J:=f^{-\hat{n}-\ell} \iota_\J.$$
Then the analogue of \eqref{eq: close to v1} holds:
\begin{equation}\label{eq: close to w1}
H^{s,*}_{\widetilde{\iota}_\J,p}\A_*^{\hat{n}+\ell}(\iota_\J)H^{u,*}_{\iota_0,\iota_{\J}}u(\A_*^{\tau(\J)}(\hat{\iota}_0)) \in \CC(w_1,{\ep}_0/2).
\end{equation}
The length of the $f^{-1}$-orbit from $\iota_\J$ to $\widetilde{\iota}_\J$ is bounded above by $k_1$.  

Having two points $\widetilde{\omega}_{\I} \in \Wloc^s(p)$ and $\widetilde{\iota}_\J\in \Wloc^u(p)$ with the desired control on the singular directions \eqref{eq: close to v1} and \eqref{eq: close to w1}, we connect their orbits near $p$ by 
$$\chi:=[\widetilde{\iota}_\J,\widetilde{\omega}_{\I}],$$
and set $\bar{\chi}:=f^{-\tau(\I)-n-\ell}\chi \in [\I]$ and $\hat{\chi}:=f^{\tau(\J)+\hat{n}+\ell}\chi \in f^{|\J|}[\J]$. 

From the construction, every edge of the rectangle $[p,\widetilde{\omega}_\I,\chi,\widetilde{\iota}_\J]_r$ is no longer than $\theta^m$.
From the choice of $m$, $H^u_{\widetilde{\omega}_{\I},\chi}\circ H^s_{p,\widetilde{\omega}_{\I}}$ is sufficiently close to the identity in that it does not move any line off itself more than $\ep_0/2$ in angle. Then from \eqref{eq: close to v1}, 
\begin{align*}
u_{\bar{\chi}}&:=\A^{n+\ell}(f^{\tau(\I)}\bar{\chi})H^u_{\omega_0,f^{\tau(\I)}\bar{\chi}}u(\A^\tau(\bar{\omega}_0))\\
&= H^u_{\widetilde{\omega}_\I,\chi} H^s_{p,\widetilde{\omega}_{\I}}u_{\bar{\omega}_\I}
\end{align*}
belongs to $\CC(v_1,\ep_0)$.

Similarly, $H^{u,*}_{\widetilde{\iota}_\J,\chi}\circ H^{s,*}_{p,\widetilde{\iota}_\J}$ doesn't move any line off itself more than $\ep_0/2$ in angle.
Notice that 
$$u(\A_*^{\tau(\J)}(\hat{\iota}_0)) = v(\A^{\tau(\J)}(\iota_0)),$$
since $\A_*(fx)$ is the transpose of the $\A(x)$. 
Then it similarly follows from \eqref{eq: close to w1} that
$$v_{\hat{\chi}}:=\A_*^{\hat{n}+{\ell}}(f^{\hat{n}+{\ell}}\chi)H^{u,*}_{\iota_0,f^{\hat{n}+{\ell}}\chi}v(\A^{\tau(\J)}(\iota_0))$$
belongs to $ \CC(w_1,\ep_0)$.

Then $u_{\bar{\chi}} \in \CC(v_1,\ep_0)$ and $v_{\hat{\chi}}\in \CC(w_1,\ep_0)$ together give the uniform angle gap (using the choice of 
$\ep_0 \in (0,\beta/8)$):
\begin{equation}\label{eq: 3beta/4}
\mangle\Big(v_{\hat{\chi}},u_{\bar{\chi}}^\perp\Big)>\frac{3\beta}{4}.
\end{equation}
\\
\textbf{Step 4.}
We use the orbit of $\chi$ to construct a connecting word $\K$.
Let $k:=2m+2\bar{\tau}+2k_1$, and note that $k$ is independent of $\I$ and $\J$. Then we define the connecting word $$\K:= [f^{|\I|}\bar{\chi}]_{\bar{k}}^w,$$ where $\bar{k}=2m+2\bar{\tau}+n+\hat{n}+2\ell$. The length of $\K$ is at most $k$. We apply Lemma \ref{lem: linear algebra 2} with 
$$A=\A^{\tau(\J)}(\iota_0),~B=H^s_{f^{\hat{n}+\ell}\chi,\iota_0} \A^{\hat{n}+\ell}(\chi),~C=\A^{n+\ell}(f^{\tau(\I)}\bar{\chi})H^u_{\omega_0,f^{\tau(\I)}\bar{\chi}},\text{ and }D=\A^{\tau(\I)}(\bar{\omega}_0):$$ recalling that $H^{s/u,*}_{x,y} = (H^{u/s}_{y,x})^*$, from \eqref{eq: 3beta/4}, such choice of $A,B,C$ and $D$ satisfies the assumption of Lemma \ref{lem: linear algebra 2} with $\theta = 3\beta/4$.
Since $C_0$ is the constant from the bounded distortion as well as the upper bound on $\|H^{s/u}_{x,y}\|$ whenever $d(x,y) \leq \theta$, we have 
\begin{align*}
\|\A(\I\K\J)\| &\geq \|\A^{\bar{k}+|\I|+|\J|}(\bar{\chi})\|,\\
&=\|\A^{\tau(\J)}(f^{\hat{n}+\ell}\chi)\A^{\hat{n}+\ell}(\chi)\A^{n+\ell}(f^{\tau(\I)}\bar{\chi})\A^{\tau(\I)}(\bar{\chi})\|,\\
&\geq C_0^{-2}\|H^s_{\hat{\chi},\hat{\iota}_0}\A^{\tau(\J)}(f^{\hat{n}+\ell}\chi)\A^{\hat{n}+\ell}(\chi)\A^{n+\ell}(f^{\tau(\I)}\bar{\chi})\A^{\tau(\I)}(\bar{\chi})H^u_{\bar{\omega}_0,\bar{\chi}}\|,\\
&=C_0^{-2}\|\A^{\tau(\J)}(\iota_0)
H^s_{f^{\hat{n}+\ell}\chi,\iota_0} \A^{\hat{n}+\ell}(\chi)
\A^{n+\ell}(f^{\tau(\I)}\bar{\chi})H^u_{\omega_0,f^{\tau(\I)}\bar{\chi}}
\A^{\tau(\I)}(\bar{\omega}_0)\|,\\
&= C_0^{-2}\|ABCD\|,\\
&\geq C_0^{-2} \sin(3\beta/4)\|A\|\|D\| \frac{m(B)^2m(C)^2}{\|B\|\|C\|},\\
&\geq C_0^{-8}\sin(3\beta/4) \frac{\varrho^{4k_1}}{\Upsilon^{2k_1}}\|\A^{\tau(\J)}(\iota_0)\|\cdot  \|\A^{\tau(\I)}(\bar{\omega}_0)\| ,\\
&\geq C_0^{-8} \sin(3\beta/4) \frac{\varrho^{4k_1+2(\bar{\tau}+m)}}{\Upsilon^{2k_1}}\|\A^{|\J|}(f^{-|\J|}\hat{\iota}_0)\|\cdot  \|\A^{|\I|}(\bar{\omega}_0)\|,\\
&\geq c\|\A(\I)\|\|\A(\J)\|,
\end{align*}
where $c:=\displaystyle C_0^{-10} \sin(3\beta/4) \frac{\varrho^{4k_1+2(\bar{\tau}+m)}}{\Upsilon^{2k_1}}$.

From the comments at the end of Step 2 as well as Remark \ref{rem: N uniform}, the constants $c$ and $k$ work in a small neighborhood of $\A$. This completes the proof.
\end{proof}

\subsection{Proof of Theorem \ref{thm: qm general}}\label{subsection: pf qm general}

The proof of Theorem \ref{thm: qm general} closely follows the proof of Theorem \ref{thm: A}. 
We will use the same notations as in the proof of Theorem \ref{thm: A} whenever applicable.

\begin{proof}[Proof of Theorem \ref{thm: qm general}]~\\\\
\textbf{Step 1.}
Let $\A_t \colon \Sig \to \text{GL}(\V_t)$, $1 \leq t \leq \kappa$ be \hol functions with uniformly continuous holonomies $H^{s/u,(t)}$ such that the pinching (A0) and the twisting (B0) conditions from Definition \ref{defn: 1-typical} hold at the common fixed point $p$ and its homoclinic point $z$. Let $P_t:=\A_t(p)$.

First, we fix the constants $\beta,\ep_0, m,\ell, \Upsilon, \varrho, C_0$ and $N$ from the proof of Theorem \ref{thm: A} such that their properties work uniformly for all $\A_t$, $t \in \{1,2,\ldots,\kappa\}$. For instance, denoting $$\beta_t:=\mangle\big(v_1^{(t)},(w_1^{(t)})^\perp\big) = \mangle\big(v_1^{(t)},\WW_1^{(t)}\big)>0,$$ let $\beta$ be the minimum of all $\beta_t$: 
$$\beta:=\min\limits_{1 \leq t \leq \kappa} \beta_t,$$
which is necessarily bounded away from 0. We define $N \in \N$ by taking the maximum among the $N$'s obtained by applying Lemma \ref{lem: turn} to $P_t$ and $\ep_0/2$ for each $1 \leq t \leq \kappa$. Similarly, other constants are chosen to work uniformly for all $\A_t$, $1 \leq t \leq \kappa$. 

For $k_1$, we re-define it as
$$k_1 := \kappa(N+\ell).$$
By further relaxing these constants, they work uniformly in a small neighborhood of each $\A_t$. 

In order to avoid overloading the super/subscripts, for the rest of the proof, we will often write $\A$ to denote $\A_t$ for some $1 \leq t \leq \kappa$ when the context is clear. Similarly, we will suppress the index $t$ from related expressions (especially from the holonomies $H^{s/u,(t)}_{x,y}$) when there is no confusion.
\\\\
\textbf{Step 2.}
Following Step 2 of the proof of Theorem \ref{thm: A}, we obtain $\bar{\tau}\in \N$ from the mixing property of $(\Sig,f)$ such that given any $\I \in \L$, there exists $\bar{\omega}_0 \in [\I] \cap \W^s(p)$ such that $f^{|\I|+\bar{\tau}}(\bar{\omega_0}) \in \Wloc^s(p)$. 
We set $$\widetilde{\omega}_0 :=\omega_0 = f^\tau \bar{\omega}_0, \text{ where } \tau=\tau(\I) := |\I|+\bar{\tau}+m.$$ 
Since $f^{|\I|+\bar{\tau}}\bar{\omega}_0$ is already on the local stable set $\Wloc^s(p)$ of $p$, we have $d(\widetilde{\omega}_0,p) \leq \theta^m$.
For each $1 \leq t \leq \kappa$, let 
$$u^{(t)}(\omega_0):=H^s_{\omega_0,p}   u(\A_t^{\tau}(\bar{\omega}_0)) \in \mathbb{P}(\V_t).$$

With $(\bar{\omega}_0,\tau)$ as the base case, we will inductively construct orbit segments $\{(\bar{\omega}_j,\tau+n_j)\}_{1 \leq j \leq \kappa}$ with $\bar{\omega}_j \in [\I]$ such that the $j$-th orbit segment $(\bar{\omega}_j,\tau+n_j)$ satisfies the following properties: setting 
\begin{equation}\label{eq: omega_j0}
\omega_j :=f^{\tau}\bar{\omega}_j \text{ and } \widetilde{\omega}_j:= f^{n_j} \omega_j,
\end{equation}
we have $ \omega_j \in \Wloc^u(\omega_0)$ and $\widetilde{\omega}_j \in \Wloc^s(p)$ with $d(\widetilde{\omega},p) \leq \theta^m$. Moreover, setting
\begin{equation}\label{eq: ut omega_j}
u^{(t)}(\omega_j):=H^s_{\widetilde{\omega}_j,p}\A^{n_j}(\omega_j)H^u_{\omega_0,\omega_j}u(\A_t^{\tau}(\bar{\omega}_0)), 
\end{equation}
we have
\begin{equation}\label{eq: ut in v_1}
u^{(t)}(\omega_j) \in \CC(v_1^{(t)},\ep_0/2) \text{ for }1 \leq t \leq j.
\end{equation}

First, we construct $\bar{\omega}_1$ similarly how we constructed $\bar{\omega}_\I$ in Step 2 of the proof of Theorem \ref{thm: A}: by applying Lemma \ref{lem: turn} to $u^{(1)}(\omega_0)$, we obtain $\widetilde{n}_0 \leq N$ such that $P_1^{\widetilde{n}_0} u^{(1)}(\omega_0)$ belongs to $\CC(v_1^{(1)},\ep_0/2)$. We then set
$$\bar{\omega}_{1} = f^{-\tau-\widetilde{n}_0}[z,f^{\widetilde{n}_0} \omega_0], \text{ and } n_1  = \widetilde{n}_0+\ell,$$
and define $\omega_1$, $\widetilde{\omega}_1$ according to \eqref{eq: omega_j0}. 
Following the same argument that established $u_{\bar{\omega}_\I}\in \CC(v_1,\ep_0/2)$ from Step 2 in the proof of Theorem \ref{thm: A}, we see that $u^{(1)}(\omega_1)$ defined in \eqref{eq: ut omega_j} belongs to $\CC(v_1^{(1)},\ep_0/2)$. This establishes \eqref{eq: ut in v_1} for $j=1$.

For the inductive step, suppose we have $(\bar{\omega}_j,\tau+n_j)$ such that \eqref{eq: ut in v_1} holds. Applying Lemma \ref{lem: turn} to $u^{(j+1)}(\omega_j)$ gives $\widetilde{n}_j \leq N$ such that 
$P_{j+1}^{\widetilde{n}_j}u^{(j+1)}(\omega_j)$ belongs to $\CC(v_i^{(j+1)},\ep_0/2)$ for some $1 \leq i \leq d_{j+1}$.
Setting $$\bar{\omega}_{j+1}:=f^{-\tau-n_j-\widetilde{n}_j}[z,f^{\widetilde{n}_j}\widetilde{\omega}_{j}] \text{ and }n_{j+1} := n_j+\widetilde{n}_j+\ell,$$
we obtain $\omega_{j+1}$ and $\widetilde{\omega}_{j+1}$ according to \eqref{eq: omega_j0}. From the choice of $\ell$, we have $d(\widetilde{\omega}_{j+1},p) \leq \theta^m$.
We need to show that for such $\omega_{j+1}$, $u^{(t)}(\omega_{j+1})$ belongs to $\CC(v_1^{(t)},\ep_0/2)$ for each $1 \leq t \leq j+1$. 

The analogous calculations to \eqref{eq: using psi_p^z} and \eqref{eq: ut relation 0} show that $u^{(t)}(\omega_{j+1})$ and $u^{(t)}(\omega_{j})$ are related by
$$u^{(t)}(\omega_{j+1}) = P_t^\ell\psi_{p,z}^{(t)}R^{(t)}_{f^{-\ell}\widetilde{\omega}_{j+1}}P_t^{\widetilde{n}_j}u^{(t)}(\omega_j)$$
for each $1 \leq t \leq \kappa$.

From the inductive hypothesis as well as the choice of $\widetilde{n}_j$, it follows that $P_t^{\widetilde{n}_j}u^{(t)}(\omega_j)$ belongs to $\CC(v_1^{(t)},\ep_0/2)$ for each $1\leq t \leq j+1$. Indeed for $1\leq t \leq j$, we already have $u^{(t)}(\omega_{j+1}) \in \CC(v_1^{(t)},\ep_0/2)$ from the hypothesis, and since $v_1^{(t)}$ is the eigendirection of $P_t$ corresponding to the largest eigenvalue in modulus, $P_t^{\widetilde{n}_j}$ maps it even closer toward $v_1^{(t)}$. For $t = j+1$, the number $\widetilde{n}_j \leq N$ is chosen so that $P_{j+1}^{\widetilde{n}_{j}}u^{(j+1)}(\omega_j)$ belongs to $\CC(v_i^{(j+1)},\ep_0/2)$ for some $1 \leq i \leq d_{j+1}$.

Since $R^{(t)}_{f^{-\ell} \widetilde{\omega}_{j+1}}$ does not move any line off itself more than $\ep_0/2$ in angle,
it follows from Lemma \ref{lem: lin algebra immediate} that $u^{(t)}(\omega_{j+1})$ belongs to $\CC(v_1^{(t)},\ep_0/2)$ for each $1 \leq t \leq j+1$, completing the inductive step.

\begin{figure}[h]
\caption{}
\centering
\includegraphics[width=0.7\columnwidth]{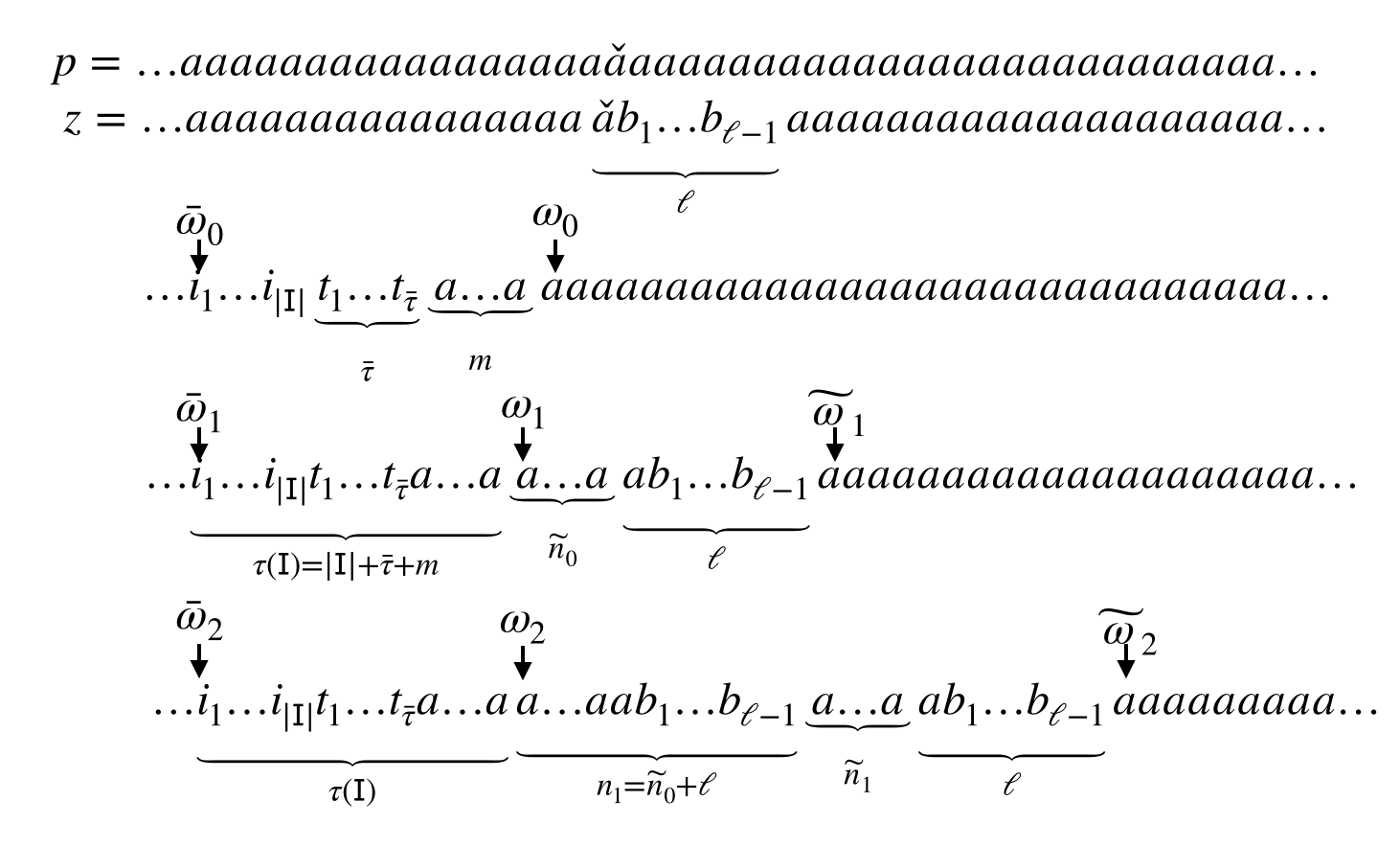}
\label{fig: figure_4}
\end{figure}

See Figure \ref{fig: figure_1} below for first two iterations of the inductive step. Also see Figure \ref{fig: figure_4} for the symbolic description of the inductive step; the checks and arrows indicate the 0-th entry of the indicated points.

\begin{figure}[h]
\caption{}
\centering
\includegraphics[width=\textwidth]{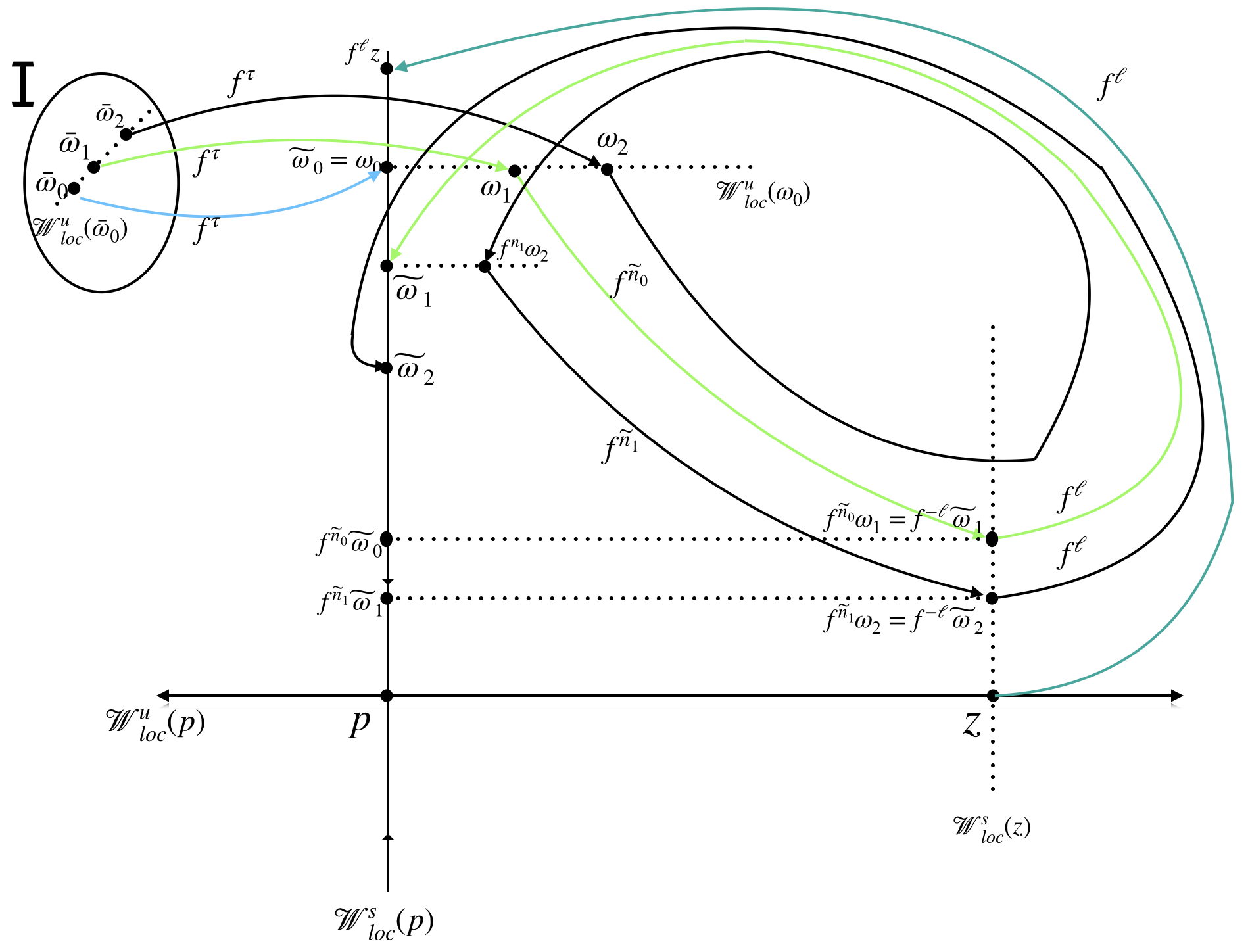}
\label{fig: figure_1}
\end{figure}

The induction ends at $\omega_\kappa$ where $\kappa$ is the number of cocycles $\A_t$. Setting $\bar{\omega}_\I:=\bar{\omega}_\kappa$ and $n_\I = n_\kappa$, we summarize the properties of the orbit segment $(\bar{\omega}_\I,\tau+n_\I)$ as follows: 
\begin{equation}\label{eq: omega_I0}
\bar{\omega}_{\I} \in [\I],~
\omega_\I:=\omega_\kappa \in \Wloc^u(\omega_0),~\widetilde{\omega}_{\I}:=\widetilde{\omega}_\kappa \in\Wloc^s(p)
\end{equation}
with $d(\widetilde{\omega}_\I,p )\leq \theta^m$, and 
\begin{equation}\label{eq: ut omega_I0}
u^{(t)}(\omega_{\I}):=u^{(t)}(\omega_\kappa)\in \CC(v_1^{(t)},\ep_0/2) \text{ for every }1 \leq t \leq \kappa.
\end{equation}
Recalling that $$k_1:=\kappa(N+\ell),$$
we have $n_\I \leq k_1$. This follows because $n_{j+1}-n_j = \widetilde{n}_j+\ell \leq N+\ell$ for each $j$ and the inductive process terminates in $\kappa$ iterations.
\\\\
\textbf{Step 3}.
We apply the same argument above to the adjoint cocycle $\A_*$ similar to Step 3 in the proof of Theorem \ref{thm: A}. 

For any $\J \in \L$, there exist $\hat{\iota}_0 \in f^{|\J|}[\J]$ such that $$\iota_0 :=f^{-\tau(\J)}\hat{\iota}_0\in \Wloc^u(p) \text{ with }\tau(\J) = |\J|+\bar{\tau}+m.$$ Similar to how we constructed $\bar{\omega}_\I$ inductively from $\bar{\omega}_0$, we  construct $\hat{\iota}_\J \in f^{|\J|}[\J]$ from $\hat{\iota}_0$ such that properties analogous to \eqref{eq: omega_I0} and \eqref{eq: ut omega_I0} hold: denoting $\iota_\J:=f^{-\tau(\J)}\hat{\iota}_\J$ and $\widetilde{\iota}_\J:=f^{-\hat{n}_\J}\iota_\J \in \Wloc^u(p)$, 
we have 
$$\iota_\J \in \Wloc^s(\iota_0) \text{ and }\widetilde{\iota}_\J \in \Wloc^u(p)  $$
with $d(\widetilde{\iota}_\J,p) \leq \theta^m$. Also, $\hat{n}_\J$ is bounded above by $k_1$. 
Moreover,
\begin{equation}\label{eq: vt in w_1}
H^{s,*}_{\widetilde{\iota}_\J,p}\A_*^{\hat{n}_\J}(\iota_\J)H^{u,*}_{\iota_0,\iota_{\J}}u^{(t)}(\A_*^{\tau(\J)}(\hat{\iota}_0)) \in \CC(w_1^{(t)},\ep_0/2) \text{ for every } 1 \leq t \leq \kappa.
\end{equation}

Having constructed two points $\widetilde{\omega}_{\I}\in \Wloc^s(p)$ and $\widetilde{\iota}_\J\in \Wloc^u(p)$ with the desired control on the singular directions \eqref{eq: ut in v_1} and \eqref{eq: vt in w_1}, we connect them near $p$ by 
$$\chi:=[\widetilde{\iota}_\J,\widetilde{\omega}_{\I}],$$
and set $\bar{\chi}:=f^{-\tau(\I)-n_\I}\chi \in [\I]$ and $\hat{\chi}:=f^{\tau(\J)+\hat{n}_\J}\chi \in f^{|\J|}[\J]$. 

From the choice of $m$, following the same argument as in Step 3 in the proof of Theorem \ref{thm: A}, we obtain the uniform angle gap:
$$\mangle\Big(v^{(t)}(\hat{\chi}),u^{(t)}(\bar{\chi})^\perp\Big)>\frac{3\beta}{4} \text{ for every }1\leq t \leq \kappa,$$
where $$ u^{(t)}(\bar{\chi}) = \A^{n_\I}(f^{\tau(\I)}\bar{\chi})H^u_{\omega_0,f^{\tau(\I)}\bar{\chi}}u(\A^\tau(\bar{\omega}_0))$$ and $$ v^{(t)}(\hat{\chi})=\A_*^{\hat{n}_\J}(f^{\hat{n}_\J}\chi)H^{u,*}_{\iota_0,f^{\hat{n}_\J}\chi}v(\A^{\tau(\J)}(\iota_0)).$$
\\\\
\textbf{Step 4.} This step also follows Step 4 in the proof of Theorem \ref{thm: A}. Setting 
$$\K:=[f^{|\I|}\bar{\chi}]^w_{\bar{k}}$$
where $\bar{k} = 2m+2\bar{\tau}+n_\I+\hat{n}_\J$, the length of $\K$ is bounded above by $k := 2m+2\bar{\tau}+2k_1$, a number defined independent of $\I$ and $\J$.
We then apply Lemma \ref{lem: linear algebra 2} to each $\A_t$, $t \in \{1,2,\ldots,\kappa\}$. This gives
$$\|\A_t(\I\K\J)\| \geq c \|\A_t(\I)\|\|\A_t(\J)\|,$$
where $c:=\displaystyle C_0^{-10} \sin(3\beta/4) \frac{\varrho^{4k_1+2(\bar{\tau}+m)}}{\Upsilon^{2k_1}}$. Here we have used the fact that all constants from Step 1 have been chosen to work uniformly over all $\A_t$. Lastly, $c$ and $k$ can be slightly relaxed to work uniformly in a small neighborhood of each $\A_t$.
\end{proof}

\begin{rem}\label{rem: nonuniform K0}
Unlike constants $c$ and $k$, it is clear from the proof of Theorem \ref{thm: qm general} that the connecting word $\K = \K(\I,\J) \in \L$ cannot be chosen uniformly in a small neighborhood of $\A$. This is because although $\B$ may be arbitrarily close to $\A$, the singular direction $u(\B_t^\tau(\bar{\omega}_0))$ from Step 2 could be drastically different from $u(\A_t^\tau(\bar{\omega}_0))$ if the length of $\I$ (and, hence, $\tau = m+\bar{\tau}+|\I|$) is arbitrarily large. Then the number of iterates $n$ of $P$ needed to turn $H^s_{\omega_0,p}u(\A_t^\tau(\bar{\omega}_0))$ close to one of the eigendirections of $P$ would be different from that of $H^s_{\omega_0,p}u(\B_t^\tau(\bar{\omega}_0))$. Hence we cannot expect $\K$ to be chosen uniformly near $\A$. 
\end{rem}

\begin{rem}\label{rem: s_0 for all R+}
From the proof of Theorem \ref{thm: qm general}, it is clear that given any $s_0 \in \R^+_0$ (i.e., not necessarily belonging to the range $[0,d]$) and $\A\in \U$, the singular value functions $\widetilde{\varphi}_\A^s$, $s \in [0,s_0]$ are simultaneously quasi-multiplicative. Moreover, the constants $c,k$ can be chosen uniformly in a small neighborhood of $\A$.

Suppose $s_0>d$. Let $c_1,k>0$ be constants from Theorem \ref{thm: qm general}, and $c_2>0$ be a small constant such that 
$$\min\limits_{x \in \Sig} \det(\B(x) )^{s/d} \geq c_2$$
for all $s \in [d,s_0]$ and $\B$ sufficiently close to $\A$.
Let $C >1$ be the maximum among the bounded distortion constant (see Remark \ref{rem: unif cts hol implies bdd dist}) of $\Phi_\B^s$ for all $s \in (d,s_0)$ and $\B$ sufficiently close to $\A$. From the definition of $\vphi^s(A):=|\det(A)|^{s/d}$ for $s >d$, simultaneous quasi-multiplicativity of $\widetilde{\vphi}_\A^s$, $s \in [0,s_0]$ holds with uniform constants $c:=\min(c_1,C^{-2}c_2^k)$ and $k$.
This remark will be useful in proving Theorem \ref{thm: B} in the following section.
\end{rem}

\subsection{Proof of Theorem \ref{thm: E}}\label{subsection: pf thm E}

From Theorem \ref{thm: qm general}, Theorem \ref{thm: E} easily follows.
\begin{proof}[Proof of Theorem \ref{thm: E}]
Given $\A \in \U$, we set $\kappa = d-1$, $\V_t = \R^{d \choose t}$ and $\A_t  = \A^{\wedge t}$ for each $1\leq t \leq d-1$. Then each $\A_t$ admits holonomies $(H^{s/u})^{\wedge t}$ where $H^{s/u} $ are the canonical holonomies of $\A$ given by the fiber-bunching assumption on $\A$.  Also, $(H^{s/u})^{\wedge t}$ varies \hol continuously because $H^{s/u}$ does from \eqref{eq: hol holder}. Hence, it follows from $\A$ being typical that $\A_t$, $1 \leq t\leq d-1$ satisfy the assumptions of Theorem \ref{thm: qm general}. 

Recalling that $\widetilde{\vphi}^t_\A = \widetilde{\vphi}^1_{\A^{\wedge t}}$ for $1\leq t\leq d-1$, Theorem \ref{thm: qm general} gives simultaneous quasi-multiplicativity of $\widetilde{\vphi}^t_\A$ when $t$ is restricted to $[1,d-1] \cap \N$. Moreover, $\widetilde{\vphi}_\A^0 \equiv 1$ is trivially quasi-multiplicative, and by decreasing $c$ if necessary (depending on $k$ and $\varrho$ from the proof of Theorem \ref{thm: qm general}), $\widetilde{\vphi}_\A^d $ is also simultaneously quasi-multiplicative with the same constants $c$ and $k$.

Simultaneous quasi-multiplicativity easily extends to include all $t \in [0,d]$ as follows: for any $t \in (n,n+1)$ with $n \in \{0,1,\ldots,d-1\}$, we write $t = n\gamma +(n+1)(1-\gamma)$ for some $\gamma \in (0,1)$. We raise the inequality from simultaneous quasi-multiplicativity of $\widetilde{\vphi}_\A^n$ by power $\gamma$:
$$\left(\widetilde{\vphi}^{n}_\A(\I\K\J)\right)^\gamma \geq \left(c \widetilde{\vphi}^{n}_\A(\I)\widetilde{\vphi}^{n}_\A(\J)\right)^\gamma.$$
Similarly, we raise the inequality from simultaneous quasi-multiplicativity of $\widetilde{\vphi}_\A^{n+1}$ by power $1-\gamma$:
$$\left(\widetilde{\vphi}^{n+1}_\A(\I\K\J)\right)^{1-\gamma} \geq \left(c \widetilde{\vphi}^{n+1}_\A(\I)\widetilde{\vphi}^{n+1}_\A(\J)\right)^{1-\gamma}.$$
Noticing that $\widetilde{\vphi}^t_\A$ is uniformly comparable to $(\widetilde{\vphi}^{n}_\A)^\gamma(\widetilde{\vphi}^{n+1}_\A)^{1-\gamma}$, multiplying the two inequalities gives simultaneous quasi-multiplicativity of $\widetilde{\varphi}_\A^t$: there exists $c_0>0$ such that
$$\widetilde{\vphi}^{t}_\A(\I\K\J) \geq c_0 \widetilde{\vphi}^{t}_\A(\I)\widetilde{\vphi}^{t}_\A(\J).$$
\end{proof}

\section{Continuity of the subadditive pressure}

\subsection{Proof of Theorem \ref{thm: B}}
In this subsection, we prove Theorem \ref{thm: B} based on the proof of Fekete's lemma. For any $\A \colon \Sig \to \glr$ and $s \in [0,\infty)$, we obtain a subadditive sequence $\{\log\alpha_n^s(\A)\}_{n \in \N}$ where 
$$\alpha_n^s(\A) :=  \sum\limits_{|\I| = n} \widetilde{\vphi}^s_{\A}(\I) = \sum\limits_{|\I| = n} \max\limits_{x \in [\I]} \vphi^s(\A^n(x)).$$
Since the base system is a subshift of finite type $(\Sig,f)$, we have $$\P(\Phi_\A^s) = \lim\limits_{n \to \infty}\frac{1}{n}\log \alpha_n^s(\A);$$
that is, we can compute the pressure by looking $(n,1)$-separated sets, and drop the limit in $\ep$ from the definition of the pressure \eqref{eq: pressure}. See Section 4 of \cite{keller1998equilibrium}.

We say that a sequence $\{a_n\}_{n \in \N}$ is \textit{almost superadditive with constant $C>0$} if for all $n,m \in \N$, we have
$$a_{n+m} \geq a_n +a_m - C.$$

In the following lemma, we use quasi-multiplicativity from Theorem \ref{thm: E} to show that given any $\A\in \U$ and $s \in [0,\infty)$, the sequence $\{\log \alpha_n^s(\B)\}$ is almost superadditive with the uniform constant $C>0$ for all $\B$ sufficiently close to $\A$.
\begin{lem}\label{lem: superadditive}
Let $\A \in \U$ and $s \in [0,\infty)$ . Then there exists $C = C_s>0$ such that the following holds: there exists a small neighborhood of $\A$ in $\U$ such that for all $\B$ in the neighborhood, the sequence $\{\log \alpha_n^s(\B)\}_{n \in \N}$ is almost superadditive with constant $C$. 
\end{lem}
\begin{proof}
There exists $C_1>0$ such that for any $n\in \N$, 
\begin{equation}\label{eq: superadditive 1}
\alpha^s_{n+1}(\A) \leq C_1 \alpha^s_{n}(\A).
\end{equation}
In fact, denoting the number of alphabets in $\Sig$ by $q$, set $C_1 = \Upsilon^s \cdot q$ where $\Upsilon = \max\limits_{x\in \Sig} \|\A(x)\|$. Increase $C_1$ slightly to ensure that \eqref{eq: superadditive 1} also holds for all $\B$ in a small neighborhood of $\A$. 

After shrinking the neighborhood if necessary, we have 
$$c \alpha_n^s(\B) \alpha_m^s(\B) \leq \sum\limits_{i=0}^k\alpha_{n+m+i}^s(\B) \leq \Big(\sum\limits_{i=0}^k C_1^i\Big) \alpha_{m+n}^s(\B)$$
where $c$ and $k$ are the uniform constants from simultaneous quasi-multiplicativity in Theorem \ref{thm: E} and Remark \ref{rem: s_0 for all R+}.
The lemma follows by setting $C =\log \Big(c^{-1} \cdot \sum\limits_{i=0}^k C_1^i\Big)$.
\end{proof}

We are now ready to prove Theorem \ref{thm: B}.

\begin{proof}[Proof of Theorem \ref{thm: B} (1)]
Let $\A \in \U$, $s \in [0,\infty)$, and $\ep>0$ be given.

First, we show that there exists $\d>0$ such that for any $\B$ sufficiently close to $\A$ and $t \in [0,\infty)$ with $|s-t|<\d$, we have
\begin{equation}\label{eq: Proof cty pressure (1)-1}
\Big| \P(\Phi_\B^s) - \P(\Phi_\B^t)\Big| <\ep/2.
\end{equation}

For any $\B$ near $\A$, consider the ratio
$$\frac{\widetilde{\vphi}_\B^s(\I)}{\widetilde{\vphi}_\B^t(\I)}, $$
for some $n\in \N$ and $\I \in \L(n)$. Suppose $x, y \in [\I]$ such that $\widetilde{\vphi}_\B^s(\I) = \vphi^s(\B^n(x))$ and $\widetilde{\vphi}^t_\B(\I) = \vphi^t(\B^n(y))$. We then write 
$$\frac{\widetilde{\vphi}_\B^s(\I)}{\widetilde{\vphi}_\B^t(\I)}  = \frac{\vphi^s(\B^n(x))}{\vphi^t(\B^n(y))} = \frac{\vphi^s(\B^n(x))}{\vphi^s(\B^n(y))} \cdot \frac{\vphi^s(\B^n(y))}{\vphi^t(\B^n(y))}.$$
Using the bounded distortion property (Lemma \ref{lem: bdd distortion} and Remark \ref{rem: unif cts hol implies bdd dist}) of $\B$, the first term in the ratio $\vphi^s(\B^n(x))/\vphi^s(\B^n(y))$ can be bounded above and below by $C_1$ and $C_1^{-1}$ for some uniform constant $C_1$ independent of $\B$ and $n$.

To bound the second term $\vphi^s(\B^n(y))/\vphi^t(\B^n(y))$ in the ratio, choose $\Upsilon $ so that it serves as an upper bound on $\max\limits_{x \in \Sig}\|\B(x)\|$ for any $\B$ sufficiently close to $\A$. If $|s-t| <\d$, then $\vphi^s(\B^n(y))/\vphi^t(\B^n(y))$ can be bounded above and below by $\Upsilon^{n\d}$ and $\Upsilon^{-n\d}$.
Then it follows from the definition of $\alpha_n^s(\B)$ that
$$\Big|\frac{1}{n}\log \alpha_n^s(\B)-\frac{1}{n}\log \alpha_n^t(\B)  \Big| \leq  \d\log \Upsilon + \frac{1}{n} \log C_1.$$
Sending $n$ to infinity, \eqref{eq: Proof cty pressure (1)-1} follows by setting $\d = \ep/(2\log\Upsilon)$.

We then show that there exists a neighborhood of $\A$ in $\U$ such that for any $\B$ in the neighborhood, 
\begin{equation}\label{eq: Proof cty pressure (1)-2}
\left|\P(\Phi_\B^s) - \P(\Phi_\A^s)  \right| <\ep/2.
\end{equation}

For any $t,n \in \N$, we write $n = qt+r$ with $0 \leq r  <t$. For all $\B$ in a small neighborhood of $\A$, Lemma \ref{lem: superadditive} gives
$$
- C\frac{ (q+1)}{n}+\frac{q}{n}\log\alpha_t^s(\B) + \frac{1}{n}\log\alpha_r^s(\B)  \leq \frac{1}{n}\log\alpha_n^s(\B) \leq \frac{q}{n}\log\alpha_t^s(\B) + \frac{1}{n}\log\alpha_r^s(\B).
$$
Notice that as $n \to \infty$, we have $q/n \to 1/t$ and $\frac{1}{n}\log \alpha_r^s(\B) \to 0$ because there are only $t$ possible values of $\alpha_r^s(\B)$. Sending $n \to \infty$,
$$  \Big| \P(\Phi_\B^s)-\frac{1}{t}\log \alpha_t^s(\B)  \Big| \leq \frac{C}{t}.$$

We choose $t \in \N$ large so that $C/t < \ep/8$. Then we shrink the neighborhood of $\A$ if necessary such that for any $\B$ in the neighborhood,
$$\Big| \frac{1}{t}\log \alpha_t^s(\A) - \frac{1}{t} \log \alpha_t^s(\B) \Big| < \ep/4.$$
Then for all $\B$ in such neighborhood of $\A$, \eqref{eq: Proof cty pressure (1)-2} follows.

Combining \eqref{eq: Proof cty pressure (1)-1} and \eqref{eq: Proof cty pressure (1)-2}, we have
$$|\P(\Phi_\A^s) - \P(\Phi_\B^t)|<\ep$$ 
for any $\B$ sufficiently close to $\A$ and any $t\in [0,\infty)$ with $|s-t|<\d = \ep/(2\log \Upsilon)$.
\end{proof}
\begin{proof}[Proof of Theorem \ref{thm: B} (2)]
From Proposition \ref{prop: unique eq as gibbs}, the equilibrium state $\mu_{\A,s}$ of $\Phi_\A^s$ is unique due to quasi-multiplicativity of $\Phi_\A^s$. Together with the continuity the map $(\A,s) \mapsto \P(\Phi_\A^s)$ on $\U \times [0,\infty)$, it follows that $\mu_{\A,s}$ also varies continuously on $\U \times [0,\infty)$. 

Indeed, suppose $(\A_n ,s_n) \in \U\times [0,\infty)$ converges to $(\A,s) \in \U \times  [0,\infty)$. By passing to a subsequence, let $\nu$ be any weak-$*$ limit of $\mu_{\A_n,s_n}$. We recall that two maps $\mu \mapsto h_\mu(f)$ and $(\Phi,\mu)  \mapsto \F(\Phi,\mu)$ are upper semi-continuous; the entropy map is upper semi-continuous from the expansivity of the base system $(\Sig,f)$, and $\F$ is upper semi-continuous from being an infimum of continuous functions. From Theorem \ref{thm: B} (1), $\nu$ must be an equilibrium state of $\Phi_\A^s$:
\begin{align*}
\P(\Phi_\A^s) = \lim\limits_{n \to \infty}\P(\Phi_{\A_n}^{s_n}) &= \lim\limits_{n \to \infty} h_{\mu_{\A_{n},s_n}}(f)+\F(\Phi_{\A_n}^{s_n},\mu_{\A_n,s_n}),\\
&\leq h_\nu(f) + \F(\Phi_\A^s,\nu).
\end{align*} 
Since $\A\in \U$, the equilibrium state $\mu_{\A,s}$ of $\Phi_\A^s$ is unique. Hence $\nu = \mu_{\A,s}$, as desired.
\end{proof}

\subsection{Applications in dimension theory and proof of Theorem \ref{thm: C}}
Theorem \ref{thm: B} has applications in the dimension theory of fractals. More specifically, we consider repellers of expanding maps. 
Let $M$ be a $d$-dimensional Riemannian manifold, and $h \colon M \to M$ be a $C^{1}$ map. 
\begin{defn}\label{defn: repeller}
A compact $h$-invariant subset $\Lambda \subset M$ is a \textit{repeller} if 
\begin{enumerate}
\item $h$ is expanding on $\Lambda$: there exists $\lambda>1$ such that $$\|D_xh(v)\| \geq \lambda\|v\|$$
for all $x \in \Lambda$ and $v \in T_xM$; 
\item there exists a bounded open neighborhood $V$ of $\Lambda$ such that 
$$\Lambda = \{x \in V \colon h^nx \in V \text{ for all } n\geq 0\}.$$
\end{enumerate}
\end{defn}

For any repeller $\Lambda$ and $s \in [0,d]$, we associate a subadditive sequence $\Phi_\Lambda^s = \{\log\vphi^s_{\Lambda,n}\}_{n \in \N}$ on $\Lambda$ where
$$ \vphi^s_{\Lambda,n}(x) :=  \vps((D_xh^n)^{-1}).$$
Then the function $s \mapsto \P(\Phi^s_\Lambda)$ is strictly decreasing, and
the equation
$$\P(\Phi_\Lambda^s) = 0$$ 
has a unique solution (see \cite{barreira2003dimension}, \cite{falconer1994bounded}, and \cite{ban2010dimensions}) which we denote by $s(\Lambda)$. Such equation is a variation of so-called \textit{Bowen's equation} first introduced in \cite{bowen1979hausdorff}, and its unique solution often carries geometric information of the underlying object. In our case, $s(\Lambda)$ is an upper bound for the upper box dimension of $\Lambda$:
\begin{prop}\cite{ban2010dimensions}
Let $\Lambda$ be a repeller. Then $s(\Lambda)$ is an upper bound on the upper box dimension of $\Lambda$; that is,
$$\overline{\dim}_B \Lambda \leq s(\Lambda).$$
\end{prop}
Such $s(\Lambda)$ is a good candidate for estimating the Hausdorff dimension of $\Lambda$, and there are many settings in which $s(\Lambda)$ is equal to the Hausdorff dimension. See \cite{barreira2003dimension}, \cite{falconer1994bounded}, and \cite{ban2010dimensions}.

\begin{rem}\label{rem: self-affine sets}
The idea of relating the Hausdorff dimension of dynamically defined fractals to the unique zero of the pressure function applies to other settings as well, including self-affine sets described in the introduction.

Let $\mathsf{T}=\{T_i\}_{i=1}^q$ be a set of affine transformations of $\R^d$ with 
$T_i(x) = A_ix+r_i$ 
where $A_i \in\glr$ is a contraction (i.e., $\|A_i\|<1$) and $r_i \in \R^d$ is a translation vector. Then there exists a unique self-affine attractor $X \subset \R^d$ invariant under $\mathsf{T}$ in the sense that
$$X = \bigcup\limits_{i=1}^q T_iX;$$
see \cite{hutchinson1981fractals}. 
Let $F_\A$ be the locally constant cocycle over the one-sided shift $(\Sigma^+_q,f)$ on $q$ alphabets generated by $\A(x) = A_{x_0}$. Then the unique zero $s(X)$ of the function
$$s \mapsto \P(\Phi_\A^s)$$
is an upper bound of, and often equal to, the Hausdorff dimension of $X$.

Indeed, Falconer \cite{falconer1988hausdorff} showed that when $\|A_i\|<1/3$ for each $i$, then for Lebesgue almost all translation vectors $(r_1,\ldots,r_q)$, the upper bound $s(X)$ is in fact equal to the Hausdorff dimension of $X$. Solomyak \cite{solomyak1998measure} then relaxed the assumption $\|A_i\|<1/3$ to $\|A_i\|<1/2$.

Moreover, B\'ar\'any, Hochman, and Rapaport \cite{barany2017hausdorff} recently considered self-affine sets $X\subset R^2$ satisfying the strong open set condition. They showed that under mild and checkable conditions, $s(X)$ is equal to the Hausdorff dimension of $X$.
\end{rem}

From the structural stability of hyperbolic sets, 
for any $C^1$-small perturbation $g$ of $h$, there exists a \textit{continuation} $\Lambda_g$ of $\Lambda$ such that $h|_\Lambda$ is conjugate to $g|_{\Lambda_g}$. In particular, $\Lambda_g$ is also a repeller (with respect to $g$). Notice from its definition (i.e., from the subadditivity of $\Phi_\Lambda^s$) that $s(\Lambda_g)$ varies upper semi-continuously in $g$.

We will now prove Theorem \ref{thm: C} by applying Theorem \ref{thm: B}. First, we introduce the analogue of the fiber-bunching condition on $\Lambda$.
\begin{defn}
Suppose $\Lambda$ is a repeller defined by $h$. For $\alpha\in (0,1)$, we say $h|_\Lambda$ is \textit{$\alpha$-bunched} if 
$$\|(D_xh)^{-1}\|^{1+\alpha} \cdot \|D_xh\| <1,$$
for all $x\in \Lambda$. 
\end{defn}

\begin{rem}
A natural class of $\alpha$-bunched repellers are small perturbations of conformal repellers.
\end{rem}

\begin{thm*}[Theorem \ref{thm: C}]
Let $M$ be a Riemannian manifold, and let $h \colon M \to M$ be a $C^{r}$ map with $r>1$. Suppose $\Lambda \subset M$ is a $\alpha$-bunched repeller defined by $h$ for some $\alpha\in (0,1)$ with $r-1>\alpha$. Then there exist a $C^1$-neighborhood $\VV_1$ of $h$ in $C^{r}(M,M)$ and  a $C^1$-open and $C^r$-dense subset $\mathcal{V}_2 \subset \mathcal{V}_1$ such that the map $$g \mapsto s(\Lambda_g)$$ is continuous on $\mathcal{V}_2$.
\end{thm*}

\begin{rem}\label{rem: VV_1 small}
The neighborhood $\VV_1$ is chosen such that $\Lambda$ has a continuation $\Lambda_g$ for every $g \in \VV_1$.
\end{rem}

We begin by relating the setting of Theorem \ref{thm: C} to the setting in Theorem \ref{thm: B}. It is well-known that the dynamics on any repeller can be coded by a one-sided subshift of finite type $(\Sig^+,f)$ via a Markov partition $\mathcal{R}$ of arbitrarily small diameter. See \cite{bowen1979hausdorff}, \cite{ruelle1982repellers}, \cite{ruelle1989thermodynamic} for discussions on Markov partitions and the coding of repellers into subshifts of finite type. 

Once we fix such a Markov partition $\mathcal{R}$ for $\Lambda$, there exists a continuous and surjective coding map 
$$\chi \colon \Sig^+ \to \Lambda$$
such that $\chi \circ f = h \circ \chi$. The number card$(\chi^{-1}x)$ is bounded on $\Lambda$, and there exists a residual set $\widetilde{\Lambda} \subset \Lambda$ such that every $x \in \widetilde{\Lambda}$ has a unique pre-image under $\chi$. 

We now take the natural extension $(\Sig,f)$ of $(\Sig^+,f)$, and consider its inverse $(\Sig,f^{-1})$. Recalling that $\pi \colon \Sig \to \Sig^+$ is the projection map, we define a cocycle $F_\B$ over $(\Sig,f^{-1})$ generated by
\begin{equation}\label{eq: defn of B}
\B(x)  = (D_{\chi(\pi x)}h)^{-1}.
\end{equation}

The reason why we consider $F_\B$ other than the usual derivative cocycle is because $\B^n \colon \Sig \to \glr$ is related to $\vps_{\Lambda,n} \colon \Lambda \to \glr$ in a following way: for any $x\in \Sig$ and $n \in \N$, we have
\begin{equation}\label{eq: B and Lambda}
\B^n(f^{n-1}x) =  (D_{\chi (\pi x)}h)^{-1} \ldots (D_{\chi(\pi (f^{n-1}x))}h)^{-1} =\vps_{\Lambda,n}(\chi(\pi x)). 
\end{equation}

Let $g$ be a $C^1$-small perturbation of $h$ in $C^{r}(M,M)$. If the perturbation is sufficiently small, then we may use the same Markov partition $\mathcal{R}$ of $\Lambda$ to code the dynamics of $g$ on $\Lambda_g$ via $\chi_g$, and take its natural extension. Then we realize the perturbation $h|_\Lambda$ to $g|_{\Lambda_g}$ as the perturbation of the cocycle $F_\B$ to $F_{\B_g}$ over the same subshift of finite type $(\Sig,f^{-1})$ where $\B_g(x) = (D_{\chi_g(\pi x)}g)^{-1}$. 

Now consider the typicality assumption on the cocycle $F_\B$ over $(\Sig,f^{-1})$. We fix $\theta$, the constant defining the metric on the base $\Sig$, such that $\theta < \|(D_xh)^{-1}\|$ for all $x \in \Lambda$. If $h$ is $C^{r}$ and $\alpha$-bunched for some $r>1$ and $\alpha \in (0,1)$ satisfying $r-1>\alpha$, then the corresponding cocycle $F_\B$ over $(\Sig,f^{-1})$ is also fiber-bunched. Denoting the canonical holonomies of $F_\B$ by $H^{s/u,-}$ (the minus sign in the superscript indicates that the cocycle is over $(\Sig,f^{-1})$), the local unstable holonomy $H^{u,-}$ is trivial from the definition \eqref{eq: defn of B} of $\B$: $H^{u,-}_{x,y} \equiv I$ for any $y$ in the local unstable set of $x$ with respect to $f^{-1}$. 

A homoclinic point $z$ of a fixed point $p$ in $\Sig$ corresponds to a sequence of points $\{z_n\}_{n \in \N_0} \in \Lambda$ such that $z_0 =\chi (\pi z)$, $h^{\ell}z_0 = \chi(\pi p)$ for some $\ell \in \N$ and 
\begin{equation}\label{eq: z_n Lambda}
hz_n = z_{n-1}, \text{ and } z_n \xrightarrow{n \to \infty} \chi(\pi p).
\end{equation}

Symbolically, if $ p = [\ldots aa \hat{a}aa \ldots] \in \Sig$ and $z = [\ldots a \hat{a}b_1\ldots b_{\ell-1}aa\ldots] \in \Sig$, then for each $n \in \N$, (from now on, we will drop the notation for the coding map $\chi$ between $\Sig^+$ and $\Lambda$) we have
$$z_n = [\underset{n+1}{\underbrace{a\ldots a}}b_1\ldots b_{\ell-1}aa\ldots] \in \Sig^+.$$ 
Moreover, $H^{s,-}_{z,p}$ is given by 
$$H^{s,-}_{z,p} = \lim\limits_{n \to \infty} \left[(D_{\pi p}h)^{n} (D_{ z_{n-1}}h)^{-1}\ldots (D_{ z_0}h)^{-1}\right].$$
Using the fact that $h^\ell z_0 = \pi p$, we have $H^{u,-}_{p,f^\ell z} = I$, and 
$$H^{u,-}_{p,z} = (D_{hz_0}h)^{-1} \ldots (D_{h^{\ell-1}z_0}h)^{-1} (D_{\pi p}h)^{\ell-1}.$$

Via $H^{s/u,-}$, the holonomy loop $\psi_{p}^{z,-}$ with respect to $F_\B$ over $(\Sig,f^{-1})$ is given by 
$$\psi_p^{z,-} = H^{s,-}_{z,p} \circ H^{u,-}_{p,z},$$
where $H^{s,-}_{z,p}$ and $H^{u,-}_{p,z}$ are given as in the paragraph above. We say that an $\alpha$-bunched repeller $\Lambda$ defined by $h$ (or simply $h$) is \textit{typical} if the corresponding cocycle $F_\B$ is typical over $(\Sig,f^{-1})$.

\begin{lem}\label{lem: VV_2 typical}
Let $h \colon M \to M$ be a $C^{r}$ map defining an $\alpha$-bunched repeller $\Lambda$. Then there exist a $C^1$-neighborhood $\VV_1$ of $h$ in $C^{r}(M,M)$ and a $C^1$-open and $C^r$-dense subset $\VV_2$ of $\VV_1$ such that any $g \in \VV_2$ is typical.
\end{lem}

\begin{proof} 
As mentioned in Remark \ref{rem: VV_1 small}, we begin by choosing $\VV_1$ sufficiently small so that $\Lambda$ has a continuation $\Lambda_g$ for every $g \in \VV_1$.

We code $h|_\Lambda$ using a Markov partition to a one-sided subshift $(\Sig^+,f)$ and take its natural extension $(\Sig,f)$. Then consider $F_\B$ over $(\Sig,f^{-1})$ defined as in \eqref{eq: defn of B}. By choosing $\mathcal{V}_1$ sufficiently small, we ensure that $\Lambda_g$ for every $g\in \VV_1$ can be coded by the same Markov partition.
For simplicity, we will continue to suppress the notation for the coding map $\chi_g:\Sig^+ \to \Lambda_{g}$ and write $\B_{g}(x) = (D_{\pi x}g)^{-1}$ where $\pi x$ refers to $\chi_{g}(\pi x) \in \Lambda_{g}$. 

Following Section 9 of \cite{bonatti2004lyapunov}, we will show that the pinching condition (A0) is $C^r$-dense via the claim below and briefly sketch the proof here.
\\\\
\noindent\textbf{Claim}: given any $C^r$-neighborhood $\W$ of $\VV_1$, there exists $g \in \W$ and a periodic point $p_{g} \in \Lambda_{g}$ such that $D_{p_{g}} g^{\text{per}(p_{g})}$ has simple real eigenvalues of distinct norms.
\\\\
First, notice that the lemma follows from the claim. Indeed, suppose there exists a fixed (or periodic) point $p \in \Lambda_{g_0}$ of some $g_0 \in \VV_1$ such that $D_pg_0$ has simple real eigenvalues of distinct norms. For $g$ sufficiently $C^1$-close to $g_0$, the property of having simple real eigenvalues of distinct norms persists at $D_{p_g}g$ where $p_g$ is the continuation of $p$ with respect to $g$. Denoting the corresponding fixed point in $\Sig$ by $\widetilde{p}_g$, the property of having simple real eigenvalues of distinct norms is equivalent on $D_{p_g}g$ and its inverse $\B_g(\widetilde{p}_g)=(D_{p_g} g)^{-1}$. Hence, the pinching condition (A0) on $F_{\B}$ is $C^1$-open. Moreover, it is $C^r$-dense in $\VV_1$ assuming that the claim holds. 

It is clear that the twisting condition (B0) on $F_{\B_g}$ is $C^1$-open because the canonical holonomies $H^{s/u,-}$ vary continuously in $g$.
The twisting condition is also $C^r$-dense; given any $\{z_n\}_{n \in \N_0}$ homoclinic (as in \eqref{eq: z_n Lambda}) to a periodic point whose derivative of the return map has simple real eigenvalues of distinct norms, the twisting assumption (B0) on $F_{\B_g}$ can be obtained with an arbitrarily small $C^r$-perturbation of $g$ near $z_0$. 
This is because an arbitrarily small $C^r$-perturbation of $g$ near $z_0$ only changes $(D_{z_0}g)^{-1}$ without affecting other terms in $\psi_p^{z,-}$, and the perturbation can be chosen to destroy any configuration preventing the twisting condition (B0) on $\psi_p^{z,-}$. Hence, in order to prove the lemma, it suffices to prove the claim.
\begin{proof}[Proof of claim]
Let $g_0$ be any map in $\W$. Given any fixed (or periodic) point $p \in \Lambda_{g_0}$, upon a small $C^r$-perturbation of $g_0$ near $p$, we assume that $P := D_pg_0$ has simple real eigenvalues of distinct norms except for some pairs of complex conjugate eigenvalues. Fix any sequence $\{z_n\}_{n \in \N_0}$ homoclinic to $p$ as in \eqref{eq: z_n Lambda}, and let $z$ be the corresponding homoclinic point in $\Sig$. Upon another small perturbation of $g_0$ near $z_0$, we assume stronger twisting condition (i.e., original formulation in \cite{bonatti2004lyapunov}) holds for $\psi_p^{z,-}$. From such twisting condition, it follows that there exists a small neighborhood $\mathcal{N}$ around $(\text{orbit of }z_0) \cup p$ such that any $g_0$-invariant set in $\mathcal{N}$ admits a $Dg_0$-invariant dominated splitting $E^1 \oplus \ldots \oplus E^k$ which agrees with the eigenspace splitting of $P$ at $p$.  

Denoting $p = [aa\ldots] \in \Sig^+$, consider a periodic point $x_m \in \Sig^+$ which repeats the word $ab_1\ldots b_{\ell-1}\underset{m}{\underbrace{a\ldots a}} \in \L(\ell+m)$. We denote the corresponding periodic point in $\Sig$ by $\widetilde{x}_m$.
When $m$ is sufficiently large, the orbit of $x_m$ belongs to $\mathcal{N}$.
Since the dominated splitting is robust, there exists a dominated splitting over the orbit of $x_m$ (for all sufficiently large $m$) with respect to any sufficiently small $C^r$-perturbation $g$ of $g_0$. Moreover, such splitting has the same index as the eigenspace splitting of $P$ at $p$.

Assuming $E^1 \oplus \ldots \oplus E^k$ is ordered in the decreasing norm of the eigenvalues of $P$, let $j$ be the largest index such that $E^j$ is 2-dimensional (i.e., corresponds to a pair of complex conjugate eigenvalues).  
Then consider a 1-parameter family of perturbations $g_t, t \in [0,1]$ near $p$ such that $Dg_t$ near $p$ is given by the post-composition of $Dg_0$ with a rotation $R_{t\ep}$ by angle $t\ep$ along the $E^j$-plane. Here $\ep>0$ is chosen sufficiently small so that $g_t$ remains in $\W$ for all $t \in[0,1]$. 

We can then show that given any small $\d>0$, there exists $t_0 \in [0,\d]$ and a sufficiently large $m$ such that the rotation number of $\B_{g_{t_0}}^{\ell+m}(\widetilde{x}_m)|_{E^j}$ is an integer. By an arbitrarily small $C^r$-pertubation of $g_{t_0}$ near $x_m$ preserving $E^j$, we can ensure that $\B_{g_{t_0}}^{\ell+m}(\widetilde{x}_m)|_{E^j}$ has two real and distinct eigenvalues. Repeating this process on $g_{t_0}$ and $x_m$, we inductively resolve all complex conjugate pairs of eigenvalues into real eigenvalues of distinct norms by arbitrary small $C^r$-perturbations. See Section 9 in \cite{bonatti2004lyapunov} for more details.
\end{proof}
This completes the proof of the lemma.
\end{proof}

\begin{rem}\label{rem: VV_2 typical}
The main content in the proof of Lemma \ref{lem: VV_2 typical} shows that the pinching condition (A0) is $C^r$-dense in $\VV_1$. Then we concluded that there exists a $C^1$-open and $C^r$-dense subset $\VV_2$ of $\VV_1$ such that the cocycle ${\B_g}(x) = (D_{\pi x}g)^{-1}$ over $(\Sig,f^{-1})$ is typical for every $g\in \VV_2$. From the same result, we can also conclude that there exists another $C^1$-open and $C^r$-dense subset $\widetilde{\VV}_2$ of $\VV_1$ such that the derivative cocycle $Dg$ is typical (in the sense of Definition \ref{defn: 1-typical}) for every $g\in \widetilde{\VV}_2$.
This remark will be useful in proving Corollary \ref{cor: perturb conformal Lyap spectrum} in Section 6. 
\end{rem}

Combining Theorem \ref{thm: B} and Lemma \ref{lem: VV_2 typical}, the map
$$(g,s) \mapsto \P(\Phi^s_{\B_g})$$
is continuous on $\VV_2 \times [0,\infty)$. Since $s(\Lambda)$ is the unique zero of the pressure function $s \mapsto \P(\Phi^s_\Lambda)$, the following lemma relates the pressure $\P(\Phi^s_\B)$ defined over $(\Sig,f^{-1})$ to the pressure $\P(\Phi_\Lambda^s)$ defined over $(\Lambda,h|_\Lambda)$.
\begin{lem}\label{lem: P=P}
$\P(\Phi^s_\B) = \P(\Phi_\Lambda^s)$.
\end{lem}
\begin{proof}
Let $\Phi_\Lambda^{s,+} = \{\log \varphi^{s,+}_{\Lambda,n}\}_{n \in \N}$ be a subadditive sequence on $\Sig^+$ defined by $$\varphi^{s,+}_{\Lambda,n}(x) := \vps_{\Lambda,n}(\chi (x))=\vphi^s((D_{\chi (x)}h^n)^{-1}).$$ 
Then $\P(\Phi_\Lambda^{s,+})$ and $\P(\Phi_\Lambda^{s})$ are equal, as described below.

Consider any $\mu \in \M(f)$ and $\nu \in \M(h)$ related by $\chi_* \mu = \nu$. From the fact that $\chi$ is the coding map, we have $$h_\mu(f) = h_\nu(h).$$
Indeed, we have $h_\mu(f) \geq h_{\nu} (h)$ since $(\Lambda,h|_\Lambda)$ is a factor of $(\Sig^+,f)$. For the reverse inequality, Ledrappier-Walters' relativised variational principle \cite{ledrappier1977relativised} states that for any $\nu \in \M(h)$, we have
$$ \sup\limits_{\widetilde{\mu} \colon \chi_*\widetilde{\mu} = \nu} h_{\widetilde{\mu}}(f) = h_\nu(h)+\int\limits_\Lambda h(f,\chi^{-1}(x))d\nu(x).$$ 
Since the pre-image $\chi^{-1}(x)$ is finite, $h(f,\chi^{-1}(x)) = 0$ for every $x \in \Lambda$, and we have $h_\mu(f) \leq h_\nu(h)$.

For $\chi_*\mu = \nu$, we also have $\F(\Phi^{s,+}_\Lambda,\mu) = \F(\Phi^s_\Lambda,\nu)$ from the definition of $\Phi^{s,+}_\Lambda$. Then it follows that $\P(\Phi_\Lambda^{s,+})=\P(\Phi_\Lambda^{s})$ from the subadditive variational principle \eqref{eq: var prin}. Hence, it suffices to show that $\P(\Phi_\Lambda^{s,+})$ is equal to $\P(\Phi_\B^s)$.

From the expansivity of $(\Sig,f)$, it suffices to consider $(n,1)$-separated sets in the definition of the pressure (see \cite{keller1998equilibrium}). Notice that on $(\Sig^+,f)$, a subset $E \subset \Sig^+$ is $(n,1)$-separated if any two distinct $x,y\in E$ satisfy $x_i \neq y_i$ for some $0 \leq i \leq n-1$ (i.e., $y \not\in [x]_n$).

For every $x\in \Sig^+$, we choose a point $\widetilde{x} \in \Sig$ such that $\pi \widetilde{x} = x$. Then \eqref{eq: defn of B} and \eqref{eq: B and Lambda} gives
\begin{equation}\label{eq: B = Lambda}
\vps(\B^n(f^{n-1}\widetilde{x})) = \varphi^{s,+}_{\Lambda,n}(x).
\end{equation}

We observe a simple relationship between $(n,1)$-separated sets in $(\Sig^+,f)$ and $(n,1)$-separated sets in $(\Sig,f^{-1})$.
Given any $(n,1)$-separated set $E$ in $(\Sig^+,f)$, for each $x \in E$ we choose any point $\widetilde{x} \in \Sig$ from $\pi^{-1}(x)$, and call the corresponding set $\widetilde{E} \subset \Sig$. Then $f^{n-1}\widetilde{E}$ is a $(n,1)$-separated set in $(\Sig,f^{-1})$. 
Conversely, given any $(n,1)$-separated set $\widetilde{E}$ of $(\Sig,f^{-1})$, the projection $\pi(f^{-n+1}\widetilde{E})$ is a $(n,1)$-separated set in $(\Sig^+,f)$. 

From \eqref{eq: B = Lambda}, 
$$\sup \big\{ \sum\limits_{x \in \widetilde{E}} \vps(\B^n(\widetilde{x})) \colon \widetilde{E} \text{ is }(n,1)\text{-separated in }(\Sig,f^{-1}) \big\}$$
is equal to 
$$\sup \big\{\sum\limits_{x \in E} \varphi^{s,+}_{\Lambda,n}(x) \colon E \text{ is }(n,1)\text{-separated in }(\Sig^+,f) \big\}$$
for each $n\in \N$.
Hence, the definition of the subadditive pressure \eqref{eq: pressure} gives $\P(\Phi^s_\B) = \P(\Phi_\Lambda^{s,+})$.
\end{proof}

\begin{proof}[Proof of Theorem \ref{thm: C}]
From Lemma \ref{lem: VV_2 typical}, there exist a $C^1$-neigbhorhood $\VV_1$ of $h$ in $C^r(M,M)$ and a $C^1$-open and $C^r$-dense subset $\VV_2$ of $\VV_1$ such that every $g \in \VV_2$ is typical.  Theorem \ref{thm: B} and Lemma \ref{lem: P=P} give us that the map $$(g,s) \mapsto  \P(\Phi_{\B_g}^s) = \P(\Phi_{\Lambda_g}^s)$$ is continuous on $\VV_2 \times [0,\infty)$. Hence the map $g\mapsto s(\Lambda_g)$ is continuous on $\VV_2$. 
\end{proof}

%
%
%
%
%
%
%
%
%
%

\section{Other applications of Theorem \ref{thm: E}}
\subsection{Pointwise Lyapunov spectrum and proof of Theorem \ref{thm: D}} We prove Theorem \ref{thm: D} in this subsection. 

Recall from the introduction that
$$\lambda_t(x) :=\lim\limits_{n \to \infty} \frac{1}{n}\log\vphi^t(\A^n(x)),$$
if the limit exists (See \cite{barreira2002lyapunov} for a general discussion on the pointwise Lyapunov exponent). 

We may think of $\lambda_t(x)$, if it exists, as the sum of top $t$ Lyapunov exponents of $x$.
Let
$$\vec{\lambda}(x) = (\lambda_1(x),\ldots,\lambda_d(x)),$$
if $\lambda_t(x)$ exists for each $1 \leq t \leq d$. Let
$$L_{\A}:=\{\vec{\alpha} \in \R^d \colon \vec{\alpha} = \vec{\lambda}(x) \text{ for some } x \in \Sig\}.$$

\begin{thm*}[Theorem \ref{thm: D}]
Let $\A \in \U$. Then $L_{\A}$ is a closed and convex subset of $\R^d$. 
\end{thm*}

\begin{rem}
Theorem \ref{thm: D} is a generalization of earlier works on the structure of various spectrums.
For instance, the pointwise Lyapunov exponent $\lambda_t(x)$ may be considered as a subadditive generalization of the Birkhoff average of a continuous function $\vphi$ defined as $$\overline{\vphi}(x) := \lim\limits_{n \to \infty} \frac{1}{n}(\vphi(x)+\ldots +\vphi(f^{n-1}x)),$$ if the limit exists.
For any \hol continous potential $\vphi$ over a mixing subshift of finite type, Pesin and Weiss \cite{pesin2001multifractal} showed that the spectrum of the Birkhoff average $\overline{\vphi}$ is a closed interval. 

For a class of subadditive potentials, Feng \cite{feng2003lyapunov,feng2009lyapunov} considered the pointwise top Lyapunov spectrum for locally constant cocycles over a subshift of finite type. Under the irreducibility assumption, he obtained a similar result to \cite{pesin2001multifractal} that the spectrum is a closed interval.

We prove Theorem \ref{thm: D} using Theorem \ref{thm: E} and ideas in \cite{feng2003lyapunov,feng2009lyapunov}. Theorem \ref{thm: D}  extends the result of Feng in two ways: we consider more general class of cocycles (i.e., fiber-bunched) and we consider the spectrum of all pointwise exponents $\lambda_t$ for $1 \leq t \leq d$ simultaneously as opposed to the top exponent $\lambda_1$ only. 
\end{rem}
\begin{proof}[Proof of Theorem \ref{thm: D}]
The idea is to carefully concatenate (using quasi-multiplicativity) a sequence of words such that the pointwise Lyapunov exponents exist and behave as controlled. Although this idea applies in showing both convexity and closedness of $L_{\A}$, the constructions are slightly different, and hence we divide the proof into two parts. 

For any $x\in \Sig$, the pointwise Lyapunov exponent $\vec{\lambda}(x)$  depends only on the forward trajectory $\pi x$ of $x$. For instance, any two points on the same stable set have the same pointwise Lyapunov exponents (if they exist). This can be seen from the bounded distortion on $\Phi_\A^t$ coming the existence of the canonical stable holonomy. Hence, we will focus on constructing a one-sided word $\omega^+ \in \Sig^+$ so that any $\omega\in \Sig$ with $\pi \omega = \omega^+$ has the desired pointwise Lyapunov exponents.

Throughout the proof, we denote (over all $1 \leq t \leq d$) the uniform constant from bounded distortion on $\Phi_\A^t$ by $C$, $\max\limits_{x \in \Sig}\|\A^{\wedge t}(x)\|$ by $\Upsilon$, $\min\limits_{x \in \Sig}m(\A^{\wedge t}(x))$ by $\varrho$, and simultaneous quasi-multiplicativity constant by $c \in(0,1)$. Also, similar to the proof of Theorem \ref{thm: qm general}, we always consider all $1 \leq t \leq d$ simultaneously even when it is not explicitly stated.
\\\\
\noindent\textbf{(1) $L_{\A}$ is closed}.\\
Let $\{x_i\}_{i \in \N}$ be a sequence of points in $\Sig$ such that their Lyapunov exponents exist and limit to some $\vec{\lambda}$: $$\vec{\lambda}(x_i) \xrightarrow{i \to \infty} \vec{\lambda} = (\lambda_1,\ldots,\lambda_d).$$ Replacing $x_i$ by a subsequence if necessary, fix a strictly decreasing sequence $\{\ep_i\}_{i \geq 1}$ with $\ep_i \to 0$ and assume that 
\begin{equation}\label{eq: epsilon i first}
|\lambda_t(x_i) - \lambda_t|< \ep_{i+1},
\end{equation}
for each $i \in \N$ and $ 1 \leq t \leq d$. We then fix a strictly increasing sequence $N_i \to \infty$ such that for any $i \in\N$ (serving as a common index for both $x$ and $\ep$) and $1 \leq t \leq d$, 
\begin{equation}\label{eq: epsilon i}
\Big| \frac{1}{N}\log \vphi^t(\A^N(x_i)) - \lambda_t(x_i)\Big| < \ep_{i+1} ~\text{ for each }N \geq N_i.
\end{equation}

Suppose we have chosen another sequence $m_i \to \infty$ with 
$m_i \gg N_{i+1}$ for each $i\in \N$ that satisfies a few extra properties to be determined below.
Define 
$$\omega^+:=[x_1]_{m_1}^w \K_1 [x_2]_{m_2}^w \K_2 [x_3]_{m_3}^w \K_3 \ldots \in \Sig^+$$ where $\K_i \in \L$ is the connecting word with $|\K_i| = k_i\leq k$ given by simultaneous quasi-multiplicativity of $\Phi^t_\A,~t=1,2,\ldots,d$. 
Let $\omega$ be any point in $\Sig$ with $\pi \omega = \omega^+$.

We claim that with appropriate choices of $m_i$'s, the pointwise Lyapunov exponent $\vec{\lambda}(\omega)$ exists and is equal to $\vec{\lambda}$. Since $\ep_i \to 0$, in order to establish the claim, it suffices to show that for each $i\in\N$ and $1 \leq t \leq d$ that
\begin{equation}\label{eq: L closed, induction}
\Big| \frac{1}{m}\log\vphi^t(\A^m(\omega)) - \lambda_t \Big| <2\ep_{i} \text{ for }\sum\limits_{j=1}^{i} m_j+ \sum\limits_{j=1}^{i-1}k_j \leq m < \sum\limits_{j=1}^{i+1} m_j+ \sum\limits_{j=1}^{i}k_j. 
\end{equation}

Consider any $m_1 \in \N$ with $m_1 \gg N_2$. For any $m = m_1+a$ with $0 \leq a < k_1+N_2$, \eqref{eq: epsilon i first} and \eqref{eq: epsilon i} give
\begin{align}\label{eq: m_1 upper}
\begin{split}
\frac{1}{m}\log \vphi^t(\A^m(\omega))  &< \frac{1}{m}\Big( \log\vphi^t(\A^{m_1}(x_1))+\log C +a\log\Upsilon \Big),\\
&\leq \frac{1}{m_1+a}\Big( m_1\lambda_t+2m_1\ep_2 +\log C + a \log \Upsilon  \Big).
\end{split}
\end{align}
For the lower bound, we similarly have
\begin{equation}\label{eq: m_1 lower}
\frac{1}{m_1+a}\Big(m_1\lambda_t -2m_1\ep_2-\log C+ a\log \varrho \Big) \leq \frac{1}{m}\log \vphi^t(\A^m(\omega)).
\end{equation}

Since $\ep_2< \ep_1$ and $a$ is bounded above by $k_1+ N_2$, if we choose $m_1$ sufficiently large, the upper bound \eqref{eq: m_1 upper} is bounded above by $\lambda_t +2\ep_1$ for all $0 \leq a < k_1+N_2$. Likewise, the lower bound \eqref{eq: m_1 lower} is bounded below by $\lambda_t-2\ep_1$ for all $0 \leq a < k_1+N_2$. This establishes \eqref{eq: L closed, induction} for $m \in [m_1,m_1+k_1+N_2)$.

Now consider $m = m_1+k_1+a$ with $a \geq N_2$ (and bounded above by $m_2$ to be chosen). We obtain different bounds on $\frac{1}{m}\log\vphi^t(\A^m(\omega))$ by using \eqref{eq: epsilon i first} and \eqref{eq: epsilon i} for $i = 2$ on the last $a$ terms in the product $\A^m(\omega)$:
\begin{align}\label{eq: m_1+k+N_2 upper}
\begin{split}
\frac{1}{m}\log\vphi^t(\A^m(\omega)) &\leq \frac{1}{m}\Big(\log\vphi^t(\A^{m_1}(x_1))+2\log C+k_1\log\Upsilon+\log\vphi^t(\A^{a}(x_2))\Big),\\
&\leq \frac{1}{m_1+k_1+a}\Big(\lambda_t(m_1+a)+2(m_1\ep_2+a\ep_3) +2\log C+ k_1\log\Upsilon \Big).
\end{split}
\end{align}	
Similarly using quasi-multiplicativity of Theorem \ref{thm: E}, we get
\begin{equation}\label{eq: m_1+k+N_2 lower}
\frac{1}{m_1+k_1+a}\Big(\lambda_t(m_1+a) - 2(m_1\ep_2+a\ep_3)-2\log C +\log c \Big)\leq \frac{1}{m}\log \vphi^t(\A^m(\omega)).
\end{equation}
We further increase $m_1$ if necessary so that the upper and lower bounds \eqref{eq: m_1+k+N_2 upper} and \eqref{eq: m_1+k+N_2 lower} still belong to $(\lambda_t -2\ep_1,\lambda_t + 2\ep_1)$ at $m=m_1+k_1+N_2$. Since the upper \eqref{eq: m_1+k+N_2 upper} and lower \eqref{eq: m_1+k+N_2 lower} bounds limit to $\lambda_t \pm 2\ep_3$ as $a \to \infty$, this gives \eqref{eq: L closed, induction} for $m \in [m_1+k_1+N_2,m_1+k_1+m_2)$, once we choose $m_2$ in the following paragraph.

We now describe the choice of $m_2 \in \N$ satisfying the following properties. As pointed out in the previous paragraph, the upper \eqref{eq: m_1+k+N_2 upper} and lower \eqref{eq: m_1+k+N_2 lower} bounds limit to $\lambda_t \pm 2\ep_3$ as $a \to\infty$. So, we choose $m_2 \gg N_3$ sufficiently large such that the upper \eqref{eq: m_1+k+N_2 upper} and lower \eqref{eq: m_1+k+N_2 lower} bounds at $m = m_1+k_1+m_2$ are close enough to $\lambda_t+2\ep_3$ and $\lambda_t-2\ep_3$, respectively. By doing so, we ensure that the upper bound
$$\frac{1}{m}\Big(\lambda_t(m_1+m_2)+2(m_1\ep_2 +m_2\ep_3)+2\log C +(a+k_1)\log \Upsilon \Big)$$
and the lower bound
$$\frac{1}{m}\Big( \lambda_t(m_1+m_2) - 2(m_1\ep_2 +m_2\ep_3) - 2\log C +\log c +a \log \varrho \Big)$$
of $\frac{1}{m}\log\vphi^t(\A^m(\omega))$ both belong to $(\lambda_t-2\ep_2,\lambda_t+2\ep_2)$ for $m = m_1+k_1+m_2+a$ with $ 0\leq a < k_2+N_3$. From the construction, \eqref{eq: L closed, induction} now holds for $m$ in the range $[m_1+k_1+m_2,m_1+k_1+m_2+k_2+N_3)$. 
Similar to \eqref{eq: m_1+k+N_2 upper} and \eqref{eq: m_1+k+N_2 lower}, $\frac{1}{m}\vphi^t(\A^m(\omega))$ for $m = m_1+k_1+m_2+k_2+a$ with $a \geq N_3$ admits the following upper and lower bound by using \eqref{eq: epsilon i first} and \eqref{eq: epsilon i} for $i=3$ on the last $a$ terms:  
$$\frac{1}{m}\Big(\lambda_t(m_1+m_2+a)+2(m_1 \ep_2+m_2 \ep_3+a\ep_4)+3 \log C + (k_1+k_2) \log \Upsilon)\Big)$$
and
$$\frac{1}{m}\Big(\lambda_t(m_1+m_2+a)-2(m_1\ep_2+m_2\ep_3+a\ep_4)-3\log C +2\log c \Big) .$$ 
We further increasing $m_2$ if necessary such that these bounds at $m = m_1+k_1+m_2+k_2+N_3$ belong to $(\lambda_t-2\ep_2,\lambda_t+2\ep_2)$.
Since these bounds limit to $\lambda_t \pm 2\ep_4$ as $a \to\infty$, this gives \eqref{eq: L closed, induction} for $m \in [m_1+k_1+m_2+k_2+N_3,m_1+k_1+m_2+k_2+m_3)$, once we choose $m_3$.


We continue this inductive process of choosing $m_i$ so that \eqref{eq: L closed, induction} holds. Similar to how we chose $m_2$, we choose $m_i \in \N$ sufficiently large such that the upper and lower bounds (obtained similar to \eqref{eq: m_1+k+N_2 upper} and \eqref{eq: m_1+k+N_2 lower}) of $\frac{1}{m}\log\vphi^t(\A^{m}(\omega))$ at $m =  \sum\limits_{j=1}^{i} m_j+ \sum\limits_{j=1}^{i-1}k_j$ are close enough to $\lambda_t\pm 2\ep_{i+1}$. In estimating $\frac{1}{m}\log\vphi^t(\A^{m}(\omega))$, the large magnitude of $m_i$ helps compensate for the next $k_i+N_{i+1}$ terms following $\sum\limits_{j=1}^{i} m_j+ \sum\limits_{j=1}^{i-1}k_j$ which only admit crude bounds using $\Upsilon$ and $\varrho$. This ensures that $\frac{1}{m}\log\vphi^t(\A^{m}(\omega))$ remains in the range of $(\lambda_t-2\ep_{i},\lambda_t+2\ep_{i})$ for all $m =\sum\limits_{j=1}^{i} m_j+ \sum\limits_{j=1}^{i-1}k_j+a$ with $0 \leq a <k_i+N_{i+1}$. For $m =\sum\limits_{j=1}^{i} m_j+ \sum\limits_{j=1}^{i}k_j+a$ with $a \geq N_{i+1}$, we use \eqref{eq: epsilon i first} and \eqref{eq: epsilon i} on the last $a$ terms with $\ep_{i+2}$, and further increase $m_i$ if necessary such that \eqref{eq: L closed, induction} remains to hold up to $m = \sum\limits_{j=1}^{i+1} m_j+ \sum\limits_{j=1}^{i}k_j$ for some $m_{i+1}$ to be chosen. Repeating this construction, we have \eqref{eq: L closed, induction} for all $m \geq m_1$, proving the claim.\\

\noindent\textbf{(2) $L_{\A}$ is convex}.\\
Let $x,y \in \Sig$ with $\vec{\lambda}(x) = \vec{\alpha}$ and $\vec{\lambda}(y) = \vec{\beta}$. We will show that for all $\gamma \in [0,1]$, there exists $\omega \in \Sig$ with $\vec{\lambda}(\omega) =\gamma\vec{\alpha} +(1-\gamma)\vec{\beta}$; the proof will construct $\omega^+ \in \Sig^+$ by concatenating the words $[x]_n^w$ and $[y]_n^w$ with proportions $\gamma$ and $1-\gamma$, respectively.

We begin by defining a sequence $\{N_i\}_{i \in \N} $ of integers given by $N_i = \lfloor \gamma i \rfloor$ if $i$ is odd and $N_i= \lfloor (1-\gamma)i \rfloor$ if $i$ is even. Then such sequence $\{N_i\}_{i \in \N}$ satisfies
\begin{equation}\label{eq: N_j properties}
\lim\limits_{i \to \infty} N_i = \infty, ~\lim\limits_{i \to \infty} \frac{(i+1)N_{i+1}}{\sum\limits_{j=1}^i jN_j} = 0,~\text{ and }\lim\limits_{i \to \infty} \frac{\sum\limits_{j=1}^i (2j-1)N_{2j-1}}{\sum\limits_{j=1}^{2i} jN_j} = \gamma.
\end{equation}

In fact, the first limit is obvious from the definition of $N_i$. Using $a-1<\lfloor a \rfloor \leq a$ for any $a\in \R$, the third limit follows because both the lower and upper bounds from  
$$\frac{\gamma\sum\limits_{j=1}^i(2j-1)^2 - \sum\limits_{j=1}^i(2j-1)}{\gamma\sum\limits_{j=1}^i(2j-1)^2 +(1-\gamma)\sum\limits_{j=1}^i(2j)^2} \leq \frac{\sum\limits_{j=1}^i (2j-1)N_{2j-1}}{\sum\limits_{j=1}^{2i} jN_j}  \leq \frac{\gamma\sum\limits_{j=1}^i(2j-1)^2}{\gamma\sum\limits_{j=1}^i(2j-1)^2 +(1-\gamma)\sum\limits_{j=1}^i(2j)^2 - \sum\limits_{j=1}^{2i}j}$$
converge to $\gamma$.
Similarly, the second limit also follows along the same reasoning.


Let $\{\omega_n\}_{n \in \N}$ be a sequence of words defined as follows:
$$\underset{N_1}{\underbrace{[x]_1^w,\ldots,[x]_1^w}},\underset{N_2}{\underbrace{[y]_2^w,\ldots,[y]_2^w}},\underset{N_3}{\underbrace{[x]_3^w,\ldots,[x]_3^w}},\underset{N_4}{\underbrace{[y]_4^w,\ldots,[y]_4^w}},\ldots;$$
that is, $\omega_i = [x]_1^w$ for $1\leq i \leq N_1$, $\omega_i = [y]_2^w$ for $N_1+1 \leq i \leq N_1+N_2$, and so on.
  
Consider 
$$\omega^+: =\omega_1 \K_1 \omega_2 \K_2 \omega_3 \K_3 \ldots \in \Sig^+$$
where each connecting word $\K_i \in \L$ with $|\K_i|=k_i \leq k$ is given by simultaneous quasi-multiplicativity from Theorem \ref{thm: E}. 

We will show that $\lim\limits_{m \to \infty} \frac{1}{m}\log \vphi^t(\A^m(\omega)) = \gamma\alpha_t +(1-\gamma)\beta_t$ for all $1\leq t \leq d$. First choose $\ep_m \to 0$ such that for each $1 \leq t\leq d$ and $m \in \N$,
$$\Big| \frac{1}{m} \log\vphi^t(\A^m(x)) - \alpha_t \Big| <\ep_m \text{ and }\Big| \frac{1}{m} \log\vphi^t(\A^m(y)) - \beta_t \Big| <\ep_m.$$
Consider any $m \in \N$ with 
\begin{equation}\label{eq: m range}
m=\sum\limits_{j=1}^{i} jN_j+\sum\limits_{j=1}^{N_1+\ldots+N_i} k_j+a~\text{ with }~ 0 \leq a<(i+1)N_{i+1}+ k_{N_1+\ldots+N_i+1}+\ldots+ k_{N_1+\ldots+N_{i+1}}.
\end{equation}
Denoting $r_j = \alpha_t$ for $j$ odd and $r_j = \beta_t$ for $j$ even, we have
\begin{align*}
\frac{1}{m}\log\vphi^t(\A^m(\omega)) &\leq \frac{1}{m}\Big(\sum\limits_{j=1}^i jN_j(r_j+\ep_j)+\log \Upsilon\big(\sum\limits_{j=1}^{N_1+\ldots+N_i} k_j+a\big)+\log C\big( \sum\limits_{j=1}^i N_j\big) \Big) ,\\
&\leq \frac{\sum\limits_{j=1}^i jN_j(r_j+\ep_j)}{\sum\limits_{j=1}^{i} jN_j}+\frac{\log\Upsilon\big((i+1)N_{i+1}+\sum\limits_{j=1}^{N_1+\ldots+N_{i+1}} k_j\big)}{\sum\limits_{j=1}^{i} jN_j} +\frac{\log C\big(\sum\limits_{j=1}^i N_j\big)}{\sum\limits_{j=1}^{i} jN_j}.
\end{align*}
Sending $m$ to $\infty$, the last two terms both limit to 0 from the definition of $N_j$ and \eqref{eq: N_j properties}. The first term limits to $\gamma \alpha_t +(1-\gamma )\beta_t$ from the third property of \eqref{eq: N_j properties} and the fact that $\ep_j \to 0$. Hence, $$\limsup\limits_{m \to \infty}\frac{1}{m}\log\vphi^t(\A^m(\omega)) \leq \gamma \alpha_t +(1-\gamma )\beta_t ~\text{ for each}~1 \leq t\leq d.$$ 

Conversely, for $m$ in the same range \eqref{eq: m range}, we obtain from simultaneous quasi-multiplicativity that 
$$\frac{1}{m}\log \vphi^t(\A^m(\omega)) \geq \frac{1}{m}\Big(\sum\limits_{j=1}^i jN_j(r_j+\ep_j)+\log c\big(\sum\limits_{j=1}^i N_j\big)+a\log\varrho - \log C \big(\sum\limits_{j=1}^i N_j\big) \Big).$$
It then follows from \eqref{eq: N_j properties} that this lower bound also limits to $\gamma \alpha_t +(1-\gamma )\beta_t$ as $m$ tends to $\infty$. Hence we have constructed $\omega^+ \in \Sig^+$ such that $\vec{\lambda}(\omega)$ exists and is equal to $\gamma\vec{\alpha} +(1-\gamma)\vec{\beta}$ for any $\omega\in\Sig$ with $\pi \omega = \omega^+$. This completes the proof.
\end{proof}

\begin{rem}\label{rem: L_{A,n}}
For each $1  \leq t \leq d$, let
\begin{equation}\label{L_{A,t}}
\vec{\lambda}_{t}(x):= (\lambda_1(x),\ldots,\lambda_t(x)),
\end{equation}
if each $\lambda_i$ exists. Note $\vec{\lambda}_d(x)$ is equal to $\vec{\lambda}(x)$. 

Then the same proof of Theorem \ref{thm: D} shows that the $t$-th pointwise Lyapunov spectrum $L_{\A,t}$ is also closed and convex for any $\A \in \U$.
\end{rem}

\begin{proof}[Proof of Corollary \ref{cor: perturb conformal Lyap spectrum}]
Fix any $\alpha \in (0,1)$ such that $r-1>\alpha$. Since $h|_\Lambda$ is conformal, by choosing $\VV_1$ sufficiently small, we ensure that any $g \in \VV_1$ is $\alpha$-bunched. From Lemma \ref{lem: VV_2 typical} and Remark \ref{rem: VV_2 typical}, there exists a $C^1$-open and $C^r$-dense subset $\VV_2$ of $\VV_1$ such that the derivative cocycle $Dg$ of any $g\in \VV_2$ is typical. Then Theorem \ref{thm: D} gives that $L_{g}$ is closed and convex.
\end{proof}

\subsection{Multifractal analysis}
Using simultaneous quasi-multiplicativity of $\Phi_\A^s$ for $\A \in \U$, we perform partial multifractal analysis of the $\vec{\alpha}$-level set 
$$E(\vec{\alpha}):=\{x \in \Sig \colon \vec{\lambda}_t(x) = \vec{\alpha}\}$$
for certain $\vec{\alpha} \in \R^t$. For a general introduction on the multifractal analysis, see \cite{barreira1997general}, \cite{pesin2001multifractal}, \cite{climenhaga2010multifractal}, \cite{climenhaga2014thermodynamic}, and \cite{feng2010lyapunov}.

For an arbitrary system $(X,f)$, arbitrary map $\A \colon X\to \glr$, and arbitrary vector $\vec{\alpha}$, the $\vec{\alpha}$-level set $E(\vec{\alpha})$ may be empty. Even when $E(\vec{\alpha})$ is non-empty, its structure may be irregular. With extra assumptions such as quasi-multiplicativity of the potential $\Phi_\A^s$, we can study such level set $E(\vec{\alpha})$ for certain $\vec{\alpha} \in \R^n$. 


We recall the general setting in which \cite{feng2010lyapunov} is applicable. Let $(X,f)$ be a compact metric space. For any $\vec{q} = (q_1,\ldots,q_t) \in \R^t_+$ and $\vec{\Phi} = (\Phi_1,\ldots,\Phi_t)$ where each $\Phi_i = \{\log \vphi_{i,n}\}_{n \in \N}$ is a subadditive sequence of potential on $X$, we define 
$$\vec{q}\cdot \vec{\Phi}:= \sum\limits_{i=1}^m q_i \Phi_i = \Big\{\sum\limits_{i=1}^m q_i\log \vphi_{i,n}\Big\}_{n \in \N}.$$
In what follows, let $$P_{\vec{\Phi}}(\vec{q}):= \P(\vec{q} \cdot \vec{\Phi})~\text{ and }~\F(\vec{\Phi},\mu) := \big(\F(\Phi_1,\mu),\ldots,\F(\Phi_t,\mu)\big),$$
where $\F$ is defined as in \eqref{eq: var prin}. 

Using Bowen's definition of entropy of non-compact sets \cite{bowen1973topological}, Feng and Huang showed that
\begin{prop}\cite[Theorem 4.8]{feng2010lyapunov}\label{prop: multifractal}
Suppose the entropy map of the system $(X,f)$ is upper semi-continuous. If $\vec{q}_0 \in \R^t_+$ such that $\vec{q}_0\cdot \vec{\Phi}$ has a unique equilibrium state $\mu_{\vec{q}_0}$, then the subadditive pressure $\P_{\vec{\Phi}}(\vec{q})$ is differentiable at $\vec{q}_0$ and the gradient $\nabla \P_{\vec{\Phi}}$ at $\vec{q}_0$ is equal to $\F(\vec{\Phi},\mu_{\vec{q}_0})$. Moreover, denoting $\vec{\alpha}:= \nabla \P_{\vec{\Phi}}(\vec{q}_0)$, the $\vec{\alpha}$-level set $E(\vec{\alpha})$ is non-empty and satisfies 
\begin{equation}\label{eq: multifractal 1}
h_{\text{top}}(E(\vec{\alpha})) = h_{\mu_{\vec{q}_0}}(f).
\end{equation}
\end{prop}

\begin{rem} We have only stated parts of \cite[Theorem 4.8]{feng2010lyapunov} in order to keep the proposition simple. Indeed, under the same assumptions and notations $\vec{\alpha}:= \nabla \P_{\vec{\Phi}}(\vec{q}_0)$, the topological entropy of the $\vec{\alpha}$-level set $E(\vec{\alpha})$ is also equal to other quantities:
\begin{align}\label{eq: multifractal 2}
\begin{split}
h_{\text{top}}(E(\vec{\alpha})) &= \inf\limits_{\vec{t} \in \R^t_+} \Big(  \P_{\vec{\Phi}}(\vec{t}) - \vec{\alpha} \cdot \vec{t} \Big) = \P_{\vec{\Phi}}(\vec{q}_0) - \vec{\alpha} \cdot \vec{q}_0,\\
&=\sup \{h_\mu(f) \colon \mu \in \M(f),~ \F(\vec{\Phi},\mu) = \vec{\alpha}\}.
\end{split}
\end{align} 
Barreira-Gelfert \cite{barreira2006multifractal} first obtained similar results for repellers of $C^{1+\alpha}$ maps satisfying the cone condition and bounded distortion. \cite{feng2010lyapunov} improved the result to the more general setting, described in Proposition \ref{prop: multifractal}. See also \cite{pesin2001multifractal} and \cite{fan2001recurrence} for related earlier works, establishing similar results for additive potentials.
\end{rem} 

We apply the proposition to $\vec{\Phi}_\A = (\Phi^1_\A,\ldots,\Phi^t_\A)$ for $\A \in \U$.
From Theorem \ref{thm: E}, it follows that the subadditive potential $\vec{q}_0 \cdot \vec{\Phi}_\A$ is quasi-multiplicative for any $\vec{q}_0 \in \R^t_+$. Then Proposition \ref{prop: unique eq as gibbs} gives the unique equilibrium state $\mu_{\vec{q}_0}$ of $\vec{q}_0 \cdot \vec{\Phi}_\A$. Hence, we obtain the following corollary:
\begin{cor}\label{cor: multifractal}
For any $\A \in \U$ and any $\vec{q}_0 \in \R^t_+$, the subadditive potential $\vec{q}_0 \cdot \vec{\Phi}_\A$ is quasi-multiplicative, and hence, has a unique equilibrium state $\mu_{\vec{q}_0}$. Also, \eqref{eq: multifractal 1} and \eqref{eq: multifractal 2} hold with $\vec{\alpha}:= \nabla \P_{\vec{\Phi}_\A}(\vec{q}_0)$.
\end{cor}

\bibliographystyle{amsalpha}
\bibliography{slnr_bib}

\end{document}